\documentclass[11pt]{amsart}

\usepackage{amsmath,amssymb,amsthm,tikz, nicefrac, mathrsfs,upgreek}

\usepackage{mathtools}
\usepackage{enumitem}
\usepackage[colorinlistoftodos]{todonotes}

\RequirePackage{color}
\RequirePackage[colorlinks, urlcolor=my-blue,linkcolor=my-blue,citecolor=my-blue]{hyperref}
\definecolor{my-blue}{rgb}{0.0,0.0,0.6}
\definecolor{my-red}{rgb}{0.5,0.0,0.0}
\definecolor{my-green}{rgb}{0.0,0.5,0.0}
\definecolor{nicos-red}{rgb}{0.75,0.0,0.0}
\definecolor{nicos-green}{rgb}{0.0,0.75,0.0}
\definecolor{light-gray}{gray}{0.6}
\definecolor{really-light-gray}{gray}{0.8}
\definecolor{sussexg}{rgb}{0.0,0.5,0.5}
\definecolor{sussexp}{rgb}{0.5,0.0,0.5}

\usepackage{dsfont} 
\usepackage[symbol]{footmisc}
\usepackage[utf8]{inputenc}

\usepackage{nicefrac}

\usepackage{multicol}
\PassOptionsToPackage{dvipsnames}{xcolor}
\usepackage{tikz,xcolor}

\usetikzlibrary{shapes,arrows}
\usetikzlibrary{hobby}
\usetikzlibrary{decorations.markings}

\newtheorem{theorem}{\color{my-blue}{\sc Theorem}}[section]
\newtheorem{lemma}[theorem]{\color{my-blue} \sc Lemma}
\newtheorem{proposition}[theorem]{\color{my-blue} \sc Proposition}
\newtheorem{corollary}[theorem]{\color{my-blue} \sc Corollary}
\newtheorem{conjecture}[theorem]{\color{my-blue} \sc Conjecture}

\numberwithin{equation}{section}
\theoremstyle{remark}
\newtheorem{remark}[theorem]{\color{my-blue} Remark}

\newcommand{\be}{\begin{equation}}
\newcommand{\ee}{\end{equation}}

\providecommand{\abs}[1]{\vert#1\vert}

\newcommand{\TV}[1]{{\lVert #1 \rVert}_{\normalfont
\text{TV}}}

\newcommand{\tmix}{t_{\textup{mix}}}
\newcommand{\tmixb}{t^{\prime}_{\textup{mix}}}

\newcommand{\sk}{S_k}

\newcommand{\eqpd}{\, .}
\newcommand{\eqcom}{\, ,}

\def\mix{\textup{mix}}

\def\bE{\mathbb{E}}
\def\bN{\mathbb{N}}
\def\bP{\mathbb{P}}

\def\bR{\mathbb{R}}

\def\e{\varepsilon}
\def\c{\textup{c}} 

\def\R{\bR}
\def\N{\bN}
\def\P{\bP}

\usepackage{stackengine}

\DeclareMathOperator{\Var}{Var}

\def\E{\bE}
\def\P{\bP} 

\definecolor{partcolor1}{rgb}{0.0,0.5,0.0}
\definecolor{partcolor2}{rgb}{0.0,0.5,0.0}

\definecolor{darkgreen}{rgb}{0.0,0.5,0.0}
\definecolor{darkblue}{rgb}{0.5,0.1,0.5}
\definecolor{deepblue}{rgb}{0.25,0.41,0.88}
\definecolor{nicosred}{rgb}{0.65,0.1,0.1}
\definecolor{light-gray}{gray}{0.7}
\allowdisplaybreaks[1]

\RequirePackage{datetime} 
\allowdisplaybreaks
\begin{document}
\usdate
\title[$\sk$ Shuffle Block Dynamics]
{The $\mathbf{\sk}$ Shuffle Block Dynamics}
\author{Evita Nestoridi}
\address{Evita Nestoridi, Princeton University, Stony Brook University, United States}
\email{evrydiki.nestoridi@stonybrook.edu} 

\author{Amanda Priestley}
\address{Amanda Priestley, The University of Texas at Austin, United States}
\email{amandapriestley@utexas.edu} 

\author{Dominik Schmid}
\address{Dominik Schmid, University of Bonn, Germany}
\email{d.schmid@uni-bonn.de}

\date{\today}
\begin{abstract}
We introduce and analyze the $S_k$ shuffle on $N$ cards, a natural generalization of the celebrated random adjacent transposition shuffle. In the $S_k$ shuffle, we choose uniformly at random a block of $k$ consecutive cards, and shuffle these cards according to a permutation chosen uniformly at random from the symmetric group on $k$ elements. We study the total-variation mixing time of the $S_k$ shuffle when the number of cards $N$ goes to infinity, allowing also $k=k(N)$ to grow with $N$. In particular, we show that the cutoff phenomenon occurs when $k=o(N^{\frac{1}{6}})$.
\end{abstract}
\maketitle
\vspace*{-0.6cm}
\section{Introduction} \label{sec:Introduction}

\subsection{Model and results}
When shuffling a deck of $N$ cards, our experience suggests that shuffles that involve only ``local'' moves, i.e.\ moves that significantly affect only a small number of cards, mix slower than shuffles that involve non-local moves. Random transpositions \cite{DSHA}, star transpositions \cite{EvitaStarTranspositions}, random-to-random \cite{BN} are examples of such local card shuffles that are known to shuffle a deck of $N$ cards in order $N \log N $ steps. Random adjacent transpositions are even slower, mixing in order $N^3 \log N$ steps \cite{Lacoin_2016,wilson2004mixing}. In contrast, the riffle shuffle and $k$-cycles for sufficiently large values of $k$  mix in only order $\log N$ steps~\cite{BD, BSZ}. In this paper, we introduce the $\sk$ shuffle, a model that interpolates between  local and global moves for the shuffles, i.e.\ it corresponds to adjacent random transpositions when $k=2$, and the shuffle which picks a uniform permutation in every step when $k=N$. \\

Our definition of the $\sk$ shuffle is inspired by block dynamics of other well-studied models, such as the Ising model \cite{CaputoSpectralIndependence,GuoJerrumBlockIsing,KnopfelBlockIsing, MartinelliIsingTrees,MartinelliBlockIsing, seoyeonblockising}, and other non-local dynamics such as the Swendsen-Wang model \cite{CaputoSpectralIndependence,BlancaSwendsenTrees,longswendsen-wang}, the random cluster model \cite{ganguly2020information}, and the Bernoulli-Laplace model with multiple swaps \cite{AlamedaBernouliLaplaceMultiSwap,Eskenazis_Nestoridi_2020}. In the $S_k$ shuffle,  each block of $k$ consecutive cards is assigned an independent rate $1$ Poisson clock. Whenever a clock rings, we shuffle the cards in the respective block according to a permutation chosen uniformly at random from the symmetric group on $k$ elements.
The main objective in this paper is the total-variation mixing time, $t_{\mix}(\varepsilon)$, for the $\sk$ shuffle on $N$ cards when $N$ goes to infinity; see Section~\ref{sec:prelim} for a formal definition of the respective quantities. We have the following first result.
\begin{theorem}\label{thm:Precutoff} For the $\sk$ shuffle, there exists a constant $c>0$ such that for all $k=k(N)$ with $k = o(N^{2/3})$ and all $\varepsilon \in (0,1)$,
\begin{equation}\label{eq:PrecutoffSk}
\frac{6}{\pi^2} \leq \liminf\limits_{N \to \infty}\frac{ k(k^2-1) \cdot \tmix(\varepsilon)}{N^2\log N}  \leq \limsup\limits_{N \to \infty}\frac{ k(k^2-1) \cdot \tmix(\varepsilon)}{N^2\log N} \leq c \eqpd
 \end{equation}
\end{theorem}
A crucial observation is that the $\sk$ shuffle treats cards differently depending on their positions in the deck.  One might notice that cards in the first and last $k$ positions of the deck move more slowly than those in the middle. In particular, following the first card of the deck, one sees that the mixing time must be at least of constant order; see Proposition~\ref{pro:Comparison} for a precise statement. To redeem this effect, we introduce extra moves in positions $1$ through $k-1$, as well as positions $N-k+1$ through $N$, and we refer to the first and last $k$ positions as \textbf{boundary}. More precisely, for each $i \in \{2,\dots,k-1\}$ we assign a rate $\delta_i^{(k)}$, respectively a rate $\delta_{N-i+1}^{(k)}$, Poisson clock to the blocks containing the first $i$, respectively the last $i$ cards of the deck. Whenever a clock rings, we shuffle the cards in
the respective block according to a permutation chosen uniformly at random from the symmetric group on $i$ elements. We refer to this as the $\mathbf{\sk}$ \textbf{shuffle with boundaries}. The exact choice of the values $\delta_i^{(k)}$ is deferred to Section~\ref{sec:prelim}, where we formally introduce both processes; see also Section \ref{sec:Comparison} for a comparison of the two processes. While the $\varepsilon$-mixing time $\tmixb(\e)$ of the $\sk$ shuffle with boundaries still satisfies the bounds in Theorem~\ref{thm:Precutoff}, the upper bound can be improved for sufficiently slow growing $k$; see also Conjecture \ref{conj:Open} when $k=o(N^{1/2})$. 

\begin{theorem}\label{thm:cutoff} For the $\sk$ shuffle with boundaries, assuming $k=o(N^{1/6})$, the $\e$-mixing time of the $\sk$-shuffle with boundaries satisfies for all $\varepsilon \in (0,1)$
\be
\lim\limits_{N \to \infty}\frac{k(k^2-1) \cdot\tmixb(\e)}{N^2\log N} =\frac{6}{\pi^2}  \eqpd
\ee
\end{theorem}
The fact that the leading order of the mixing time does not depend on $\varepsilon$ is called the \textbf{cutoff phenomenon}. Theorem~\ref{thm:cutoff} agrees with the celebrated result of Lacoin in \cite{Lacoin_2016} where he proves the special case $k=2$; see earlier work by Wilson \cite{wilson2004mixing} for a sharp lower bound. While we follow for Theorem~\ref{thm:cutoff} the strategy introduced by Lacoin in \cite{Lacoin_2016}, one faces several challenges when adapting the arguments to the case $k\geq 3$.
Perhaps most surprisingly, neither the spectral gap nor the other eigenvalues of the transition matrix offer a simple closed form when $k>3$. Instead, under a suitable choice of the boundary rates, we utilize approximate eigenvalues and eigenfunctions. The idea of using approximate eigenfunctions first appeared in \cite{nam2019cutoff} when studying a time inhomogeneous version of the adjacent transposition shuffle. It was later adapted for continuous time Markov chains in \cite{gantert2020mixing} to get sharp lower bounds for the symmetric exclusion process with one open boundary. 
In contrast to the above mentioned works, we approximate all relevant eigenvalues and eigenfunctions of the $\sk$ shuffle with boundaries simultaneously. This allows us to perform approximate Fourier Analysis. It is a crucial ingredient in the proof of the upper bound in Theorem \ref{thm:cutoff} as it addresses a discrete heat equation whose solution can be given via the $\sk$ shuffle. 

When providing sharp lower bounds on the mixing time, another difficulty occurs for large $k$ as the maximal displacement by a shuffle within a block becomes comparable to the fluctuations of the $\sk$ shuffle in equilibrium. 
 We resolve this issue by relying on the strong Rayleigh property for the $\sk$ shuffle with boundaries similar to Salez \cite{salez2022} and Tran \cite{Tran2022} for the symmetric exclusion process with open boundaries. Moreover, we require a different generalized version of Wilson's Lemma, going back to the original second moment method as it was introduced in \cite[Section 4]{BL}, using stricter variance bounds.
Furthermore, the $\sk$ shuffle with boundaries process requires a different interpretation of censoring than the one used previously by Lacoin and Gantert et al.~ \cite{gantert2020mixing,Lacoin_2016}.  We introduce a generalized 
censoring scheme that it does not only restrict  moves, but also alters them. 
Let us conclude the introduction by mentioning that adjacent transpositions can also be studied for biased card shuffling methods; see for example \cite{BBHM:MixingBias,LingfuBiasedCardShuffling}.

\subsection{Structure of the paper}
In Section~\ref{sec:prelim}, we give preliminary definitions and notation which will be used throughout the paper. Section~\ref{sec:properties preserved} discusses important properties of the $\sk$ shuffle that are retained from the case of $k=2$,  and which play substantial roles in our proofs of the upper and lower bounds.
In Section~\ref{sec:lowerbounds}, we obtain lower bounds on the mixing time of the $\sk$ shuffle using a generalized version of the second moment method. 
In Section~\ref{sec:generalupperbound}, we introduce a coupling argument for the upper bound in Theorem \ref{thm:Precutoff}. In Section \ref{sec:lacoinupperbound}, we adapt the argument of Lacoin from \cite{Lacoin_2016} to prove a sharp upper bound on the mixing time for the $\sk$ shuffle with boundaries. We conclude by a discussion in Section \ref{sec:Comparison} comparing the $\sk$ shuffle with and without boundaries, and an open question. 

\section{Preliminaries on the $\sk$ shuffle}\label{sec:prelim}

\begin{figure}\label{fig:Skshuffle}
\centering
\begin{tikzpicture}[scale=0.6]

	\node[shape=circle,scale=1,draw] (A1) at (-3,0){$2$} ;
    \node[shape=circle,scale=1,draw] (B1) at (-1,0){$4$} ;
    \node[shape=circle,scale=1,draw,red] (C1) at (1,0){$7$} ;
	\node[shape=circle,scale=1,draw,red] (D1) at (3,0){$5$} ;
	\node[shape=circle,scale=1,draw,red] (E1) at (5,0) {$1$};
 	\node[shape=circle,scale=1,draw] (F1) at (7,0){$3$} ; 
 	\node[shape=circle,scale=1,draw] (G1) at (9,0){$6$} ; 	

		\draw[thick] (A1) to (B1);	
		\draw[thick] (B1) to (C1);
		\draw[thick] (C1) to (D1); 
		\draw[thick] (D1) to (E1);
		\draw[thick] (E1) to (F1);    
		\draw[thick] (F1) to (G1);      

	\node[shape=circle,scale=1,draw] (A2) at (-3,-2){$2$} ;
    \node[shape=circle,scale=1,draw] (B2) at (-1,-2){$4$} ;
    \node[shape=circle,scale=1,draw] (C2) at (1,-2){$5$} ;
	\node[shape=circle,scale=1,draw] (D2) at (3,-2){$1$} ;
	\node[shape=circle,scale=1,draw,red] (E2) at (5,-2) {$7$};
 	\node[shape=circle,scale=1,draw,red] (F2) at (7,-2){$3$} ; 
 	\node[shape=circle,scale=1,draw,red] (G2) at (9,-2){$6$} ; 	

		\draw[thick] (A2) to (B2);	
		\draw[thick] (B2) to (C2);
		\draw[thick] (C2) to (D2); 
		\draw[thick] (D2) to (E2);
		\draw[thick] (E2) to (F2);    
		\draw[thick] (F2) to (G2);      

	\node[shape=circle,scale=1,draw] (A3) at (-3,0-4){$2$} ;
    \node[shape=circle,scale=1,draw] (B3) at (-1,0-4){$4$} ;
    \node[shape=circle,scale=1,draw] (C3) at (1,0-4){$5$} ;
	\node[shape=circle,scale=1,draw] (D3) at (3,0-4){$1$} ;
	\node[shape=circle,scale=1,draw] (E3) at (5,0-4) {$3$};
 	\node[shape=circle,scale=1,draw] (F3) at (7,0-4){$7$} ; 
 	\node[shape=circle,scale=1,draw] (G3) at (9,0-4){$6$} ; 	

		\draw[thick] (A3) to (B3);	
		\draw[thick] (B3) to (C3);
		\draw[thick] (C3) to (D3); 
		\draw[thick] (D3) to (E3);
		\draw[thick] (E3) to (F3);    
		\draw[thick] (F3) to (G3);      
		
%
   
	\end{tikzpicture}
	\caption{\label{fig:Sk}Example of the $S_3$ shuffle on a segment of length $N$. The positions chosen in each step to be updated are marked in red.}
\end{figure}
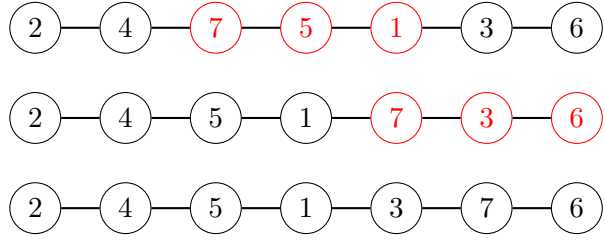

In this section, we give a formal definition of the $\sk$ shuffle with and without boundaries. For all $n \in \N$, we denote in the following by $\mathcal{S}_n$ the symmetric group on $n$ elements, and refer to $\sigma \in \mathcal{S}_N$ as a permutation on $[n]\coloneqq \{1,\dots,n\}$. For integers $i,j \in [n]$ with $i<j$, and permutations $\eta\in \mathcal{S}_n$ and $\sigma\in \mathcal{S}_{j-i+1}$, we define the configuration $\eta^{\sigma, i,j}$ by
\begin{equation}\label{eq:etaconfiguration}
    \eta^{\sigma,i,j}(m) \coloneqq \begin{cases}
    \eta(m) & \text{ if } m \notin [i,j] \cap [n] \\
    \eta(\sigma(m+1-i)) & \text{ if } m \in [i,j] \cap [n]
    \end{cases}
\end{equation} for all $m\in [n]$, i.e.\ we permute the cards in the interval $[i,j]$ according to $\sigma$. For $N \in \N$ and $k\in [N]$, the $\mathbf{\sk}$ \textbf{shuffle} on a deck of size $N$ is the continuous-time Markov chain on $\mathcal{S}_N$ whose generator is given by
\begin{equation}\label{eq:skgenwithoutboundary}
    (\mathcal{L}f)(\eta)= \sum_{ i=1}^{N-k} \frac{1}{k!} \sum_{\sigma \in S_{k}} \left( f(\eta^{\sigma,i,i+k-1}) - f(\eta)\right) 
\end{equation} 
for all functions $f \colon \mathcal{S}_N \rightarrow \R$, and $\eta \in \mathcal{S}_N$; see Figure \ref{fig:Sk} for a visualization. We denote the resulting dynamics by $(\eta_t)_{t \geq 0}$. For the $\mathbf{\sk}$ \textbf{shuffle with boundaries}, we set
\begin{equation}\label{eq:deltaWeights}
        \delta_{k-i}^{(k)}= \delta_{N-k+(i+1)}^{(k)} \coloneqq \frac{4k^2-6ik+3i^2-1}{(2(k-i)+1)(2(k-i)-1)}
\end{equation}
for $i \in [1, k-2]$ and define the dynamics $(\zeta_t)_{t \geq 0}$ on $\mathcal{S}_N$ with respect to the generator
\begin{align}\label{eq:skgenwithboundary}
    (\tilde{\mathcal{L}}f)(\zeta) \coloneqq (\mathcal{L}f)(\zeta) +  \sum_{ i=2}^{k-1} \frac{\delta_{i}^{(k)}}{i!} \sum_{\sigma \in S_{i}} \left( f(\zeta^{\sigma,1,i}) +  f(\zeta^{\sigma,N-i+1,N}) - 2 f(\zeta)\right) \, .
\end{align}
In words, we in addition apply a uniform permutation on the first and last $i$ cards at rate $\delta_{i}^{(k)}$, respectively. While the choice of $\delta_{i}^{(k)}$ may seem slightly unnatural at first glance, we will see that this choice of rates in \eqref{eq:deltaWeights} allows us to reuse the  eigenvalues and eigenfunctions for $k=2$ as approximate eigenvalues and approximate eigenfunctions for $k \geq 3$.

Note that both dynamics are reversible with respect to the uniform measure on $\mathcal{S}_N$, which we denote in the following by $\mu_N$. We let, for a probability measure $\nu$ on $\mathcal{S}_N$, 
\begin{equation}\label{def:TVDistance}
\TV{ \nu - \mu_N } \coloneqq \frac{1}{2}\sum_{\eta \in \mathcal{S}_N} \abs{\nu(\eta)-\mu_N(\eta)} = \max_{A \subseteq \mathcal{S}_N} \left(\nu(A)-\mu_N(A)\right) 
\end{equation} be the \textbf{total-variation distance} of $\nu$ and $\mu_N$, and let the $\boldsymbol\varepsilon$\textbf{-mixing time} of $(\eta_t)_{t \geq 0}$ be
\begin{equation}\label{def:MixingTime}
t_{\text{\normalfont mix}}(\varepsilon) \coloneqq \inf\left\lbrace t\geq 0 \ \colon \max_{\eta \in S_{N}} \TV{\P\left( \eta_t \in \cdot \ \right | \eta_0 = \eta) - \mu_N} < \varepsilon \right\rbrace
\end{equation} for all $\varepsilon \in (0,1)$. Similarly, we denote the  $\varepsilon$-mixing time of $(\zeta_t)_{t \geq 0}$ by $\tmixb(\e)$ for all $\e \in (0,1)$.
One central tool is the \textbf{height function} of the $\sk$ shuffle. For $\sigma \in \mathcal{S}_N$, we set
\be\label{eq:heightfunction}
h_{\sigma}(x, y)\coloneqq \left(\sum\limits_{z=1}^{x} \mathds{1}_{\{\sigma(z)\leq y\}}\right)-\frac{x y}{N} \eqcom
\ee with the convention that $h_\sigma(x)=h_\sigma(x,\lfloor N/2 \rfloor)$
for all $x\in [N]$ and $\sigma \in \mathcal{S}_N$. Further, with a slight abuse of notation, we write
$h_t(x)=h_{\eta_t}(x,\lfloor N/2 \rfloor)$ for all $t\geq 0$ for the $\sk$ shuffle $(\eta_t)_{t\geq 0}$, and similarly $h^{\prime}_t(x)=h_{\zeta_t}(x,\lfloor N/2 \rfloor)$  for the $\sk$ shuffle with boundaries. Observe the  height functions allow one to define a partial order on the state space $\mathcal{S}_N$. We say that $\sigma$ \textbf{dominates} $\sigma^{\prime}$, and write $\sigma \succeq \sigma^{\prime}$ if, for all $x,y \in [N]$,
\begin{equation}\label{def:PartialOrder}
    h_{\sigma}(x,y) \geq h_{\sigma^{\prime}}(x,y) \eqpd
\end{equation}
Note that  the maximal element with respect to $\succeq$ is $\sigma=\textup{id}$, the identity on~$\mathcal{S}_N$.

\subsection{An approximation of the spectrum}

In the following, we discuss the spectrum of the $\sk$ shuffle with boundaries. Apart from the special cases $k=2$ and $k=3$, we shall see that the eigenfunctions and eigenvalues do not have a simple closed form, and instead we propose the following candidates as approximate eigenvalues and eigenfunctions, i.e.\ we set
\begin{align}\label{eigenfunction}
    \Phi^{(j)}_{N,y}(\sigma)\coloneqq \sum_{x=1}^{N-1} h_{\sigma}^{\prime}(x,y)\psi_j(x) \quad \text{ where } \quad \psi_j(x) \coloneqq \sin \left(\frac{ x j\pi}{N}\right) \eqcom
\end{align}
with the convention that $\Phi_N(\sigma)=\Phi^{(1)}_{N,N/2}(\sigma)$. Moreover, we let  for all $j\in [N]$
\be
\begin{split}
\label{eq:eigenvaluebound}
\lambda_{N,k}^{(j)} \coloneqq (k-1)- \left[2 \sum_{i=1}^{k} \frac{k-i}{k} \cos \left(\frac{i j\pi}{N}\right)\right]
= \frac{kj^2\pi^2}{N^2}\left(\frac{k^2-1}{12}\right) + O\Big(\frac{k^5j^4}{N^4}\Big) \, .
\end{split}
\ee
The following lemma shows that for our particular choice of  $\delta_i$ defined in \eqref{eq:deltaWeights}, $\Phi_{N}^{(j)}$ are indeed suitable approximate eigenfunctions with respect to approximate eigenvalues $\lambda^{(j)}_{N,k}$.  
\begin{lemma}\label{lem:ApproxEigen}Recall from \eqref{eq:skgenwithboundary} the definition of $\tilde{\mathcal{L}}$. Let $k=k(N)$ be such that $k = o(N)$. Then there exists some constant $C>0$ such that for all $j\in [N]$, and for all $y \in [N-1]$ 
\be\label{operatorbound}
|(-\tilde{\mathcal{L}} \Phi^{(j)}_{N,y})(\sigma)-\lambda^{(j)}_{N,k} \Phi^{(j)}_{N,y}(\sigma)| \leq  C k^6j^3N^{-3}  
\ee
 for all $\sigma \in \mathcal{S}_N$. 
\end{lemma}

\begin{proof} We consider in the following only the case $y=\frac{N}{2}$ as the remaining cases are similar.
Notice that whenever we apply a permutation in $[a,b]$ chosen uniformly at random to a configuration $\sigma$, the expected height function evaluated at a position $x\in [a,b]$ is given by $(b-x)(b-a)^{-1}h_{\sigma}(a)+(x-a)(b-a)^{-1}h_{\sigma}(b)$. Therefore, by re-indexing the summation and using the definition of $\tilde{\mathcal{L}}$, we see that 
\begin{align*}
    \begin{split}\label{generatortimesph}
         &(\tilde{\mathcal{L}} \Phi^{(j)}_N)(\sigma)= \sum_{x=1}^{N-1}(\tilde{\mathcal{L}}h)(\sigma)\psi_j(x) =\sum\limits_{x=1}^{N-1}h_{\sigma}(x)a_x
    \end{split}
\end{align*}
where we set
\begin{equation*}
a_x = \begin{cases}
    \sum\limits_{i=1}^{k-1}\frac{k-i}{k}(\psi_j(x-i)+\psi_j(x+i)) - (k-1)\psi_j(x)  &\text{if } x\in [k,N-k] \\
    \sum\limits_{i=1}^{x-1}\frac{\delta^{(k)}_{x}i}{x}\psi_j(i)+\sum_{i=1}^{k-1}\frac{k-i}{k}\psi_j(x+i)-\Big(x+\sum\limits_{i=x+1}^{k-1}\delta^{(k)}_{i}\Big)\psi_j(x) &\text{if } x < k \\
    \sum\limits_{i=x}^{N-1}\frac{\delta^{(k)}_{x}(N-
           i)}{N-x+1}\psi_j(N-i)+\sum_{i=1}^{k-1}\frac{k-i}{k}\psi_j(x-i)-\Big(x+\sum\limits_{i=x+1}^{N-1}\delta^{(k)}_{i}\Big)\psi_j(x) &\text{if } x > N- k +1 \eqpd
\end{cases}
\end{equation*}
For all $x\in [N]$, a computation involving trigonometric identities shows that
\begin{align*}
\begin{split}
\lambda^{(j)}_{N,k} \psi_j(x) =  \sum\limits_{i=1}^{k-1}\frac{k-i}{k}(\psi_j(x-i)+\psi_j(x+i)) - (k-1)\psi_j(x)  \eqpd
\end{split}
\end{align*}
By our choice of $\delta^{(k)}_i$ in \eqref{eq:deltaWeights} another computation yields that 
\begin{equation}\label{eq:Ingrid1}
 \frac{\delta^{(k)}_x}{x}\left(\sum_{i=1}^{x-1} i^2\right)  -\Big(x-k+1+\sum\limits_{i=x+1}^{k-1}\delta^{(k)}_{i}\Big)x = \sum\limits_{i=1}^{k-1}\frac{k-i}{k}(x-i) 
\end{equation} for all $x < k$, and similarly when $x>N-k+1$. Using Taylor approximation, we get 
\begin{equation}\label{eq:Ingrid2}
   \left| \psi_{j}(x) - \frac{jx\pi}{N} + \frac{j^3x^3\pi }{6N^3} \right| \leq C \frac{j^5k^5}{N^5} 
\end{equation}
for some $C>0$ for all $N$ sufficiently large. By a telescopic summation, for all $x<k$
\begin{equation}\label{eq:BoundOnDelta}
     \delta^{(k)}_{x} \leq  \frac{7k^2}{4x^2-1} \quad  \text{ and }  \quad \left( \sum\limits_{i=k-x}^{k-1}\delta^{(k)}_{i}   \right) =   \frac{x(4k-3x-1)}{4(k-x)+2} \leq \frac{x^2}{4(k-x)}+x  
\end{equation} 
and thus, for some $c_1>0$, and all $j\geq 1$ and $x<k$
\begin{equation}\label{eq:Ingrid3}
 \left|  \sum\limits_{i=1}^{x-1}\frac{\delta^{(k)}_{x}i^4j^3}{xN^3}\right| +    \left| \frac{x^3j^3}{N^3}\sum\limits_{i=x+1}^{k-1}\delta^{(k)}_{i}  \right| \leq c_1\frac{k^4j^3}{N^3} .
\end{equation}
Since  $\max(h_{\zeta}(x),h_{\zeta}(N-x))\leq x$ for all $\zeta \in \mathcal{S}_N$, we obtain from \eqref{eq:Ingrid1}, \eqref{eq:Ingrid2} and \eqref{eq:Ingrid3}  
\begin{equation}\label{eq:approximateEigenfunctionError}
 | h_{\zeta}(x)a_x - \lambda^{(j)}_{N,k} \psi_j(x) | \leq c_2 \frac{j^3k^5}{N^3} 
\end{equation} for some constant $c_2>0$, uniformly in $x\in [k-1]$ as well as $x>N-k+1$. Summing over all $x$ in the boundary, we obtain the desired result.
\end{proof}


\subsection{Projection of the $\sk$ shuffle}\label{sec:skprojection}
Note that as in \cite{Lacoin_2016} and \cite{wilson2004mixing}, the $\sk$ shuffle has a natural projection which can be seen as an exclusion process on a hypergraph. Let $\sigma \in \mathcal{S}_N$ and let $K\in [ N-1]$ be fixed. We let $\xi^K_{\sigma} \in \{0,1\}^N$ be the configuration which we obtain by setting  $\xi^K_{\sigma}(x)= 1$ if the value of the card at position $x$ is at most $K$.
In other words, the first $K$ cards can be thought of as particles, while the remaining cards are given the role of empty sites. The corresponding dynamics $(\xi_t^K)_{t\geq 0}$ can then be described by the generator: 
\begin{equation}\label{eq:skexclusion}
    (\hat{\mathcal{L}}f)(\xi)= \sum_{ i=1}^{N-k} \frac{1}{k!} \sum_{\xi \in \{0,1\}^N} \left( f(\xi^{\sigma,i,i+k-1}) - f(\xi)\right)\eqcom
\end{equation} 
where $f$ is a function $f: \{0,1\}^N \to \R$. Here, $\xi^{\sigma,i,j}$ is defined as in \eqref{eq:etaconfiguration} for $\eta^{\sigma,i,j}$.

\section{Properties Preserved by the $\sk$ Shuffle} \label{sec:properties preserved}
In this section, we discuss properties that are shared by the $\sk$ and $S_2$ shuffles. It is of great importance that the stationary distribution of the $\sk$ shuffle is the uniform distribution, allowing us to transfer several properties from the case of $k=2$ to $k\geq 3$.

\subsection{Preservation of the Censoring Inequality}\label{subsec:censoring}
The censoring inequality is introduced by Peres and Winkler in \cite{Peres_Winkler_2013} and has since been used in many contexts, such as \cite{gantert2020mixing} and  \cite{Lacoin_2016}. In this section, we define a \textbf{censoring scheme} for the $\sk$-shuffle and show that the censoring inequality holds. In contrast to the typical use of censoring, we  also alter moves.\\

Formally, we define a censoring scheme $\mathcal{C}: \R_0^+ \to \mathcal{P}(E)$  as a c\`{a}dl\`{a}g function, where $\mathcal{P}(E)$ is the power set of edges $E\coloneqq \{ \{x,x+1\} \colon x\in [N-1]\}$. We obtain the \textbf{censored dynamics} $(\eta^{\mathcal{C}}_t)_{t \geq 0}$ from the $\sk$ shuffle $(\eta_t)_{t \geq 0}$ and a censoring scheme $(\mathcal{C}_t)_{t\geq 0}$ as follows: 
Suppose that at time $t$, we perform a shuffle on an interval $\mathcal{I}=[i,j]$. If $\mathcal{I}$ contains no edge from $\mathcal{C}_t$, then we perform the shuffle on $\mathcal{I}$ as in the original dynamics. However, if $\mathcal{I}$ contains at least one edge in $\mathcal{C}_t$, then we partition $\mathcal{I}$ into sub-intervals $(\mathcal{I}_m)_{m\geq 0}$ with $\mathcal{I}_m=[i_m,i_{m+1}-1]$ such that
\begin{equation}\label{eq:Interval}
\mathcal{I} = \bigcup_{m\geq 0} [i_m,i_{m+1}-1]
\end{equation} for some $i_0<i_1<i_2<\cdots$ and that  $\{i_m-1,i_m\} \in \mathcal{C}_t$ for all $m$. In each interval $\mathcal{I}_m$, we perform an independent $S_{|\mathcal{I}_m|}$-shuffle of the elements.

In words, we obtain the censored dynamics by performing independent $S_{\cdot}$ shuffles on the sub-intervals whenever we would perform a shuffle operation along a censored edge. The censoring inequality states that the law of the censored dynamics stochastically dominates the law of the original dynamics in terms of the stochastic order $\succeq$ from \eqref{def:PartialOrder} for any time~$t \geq 0$. Here, recall that a measure $\mu$ \textbf{stochastically dominates} a measure $\nu$ on $\mathcal{S}_N$ whenever $\mu(A) \geq \nu(A)$ for any set $A \subseteq \mathcal{S}_N$ which is increasing with respect to $\succeq$; see Section~22.2 of \cite{LevinPeresWilmer}. Moreover, this stochastic domination occurs uniformly in the choice of the jump times $(\mathcal{T}^x_i)_{i\geq 1}^{x\in [N-1]}$ at which we perform the $i^{\text{th}}$ update at the interval starting at~$x$. Formally, we say that the \textbf{censoring inequality} holds, if for all $t \geq 0$ and for any suitable family $(t_i^{x})^{x\in [N-1]}_{i\in \N}$ with $t_i^{x}\geq 0$,
\begin{equation}
\P( \eta^{\mathcal{C}}_t \in \cdot \, | \,  \mathcal{T}^x_i=t_{i}^x ) \succeq \P( \eta_t \in \cdot \, | \,  \mathcal{T}^x_i=t_{i}^x )\eqpd
\end{equation} 
Recall that a function $f$ is \textbf{increasing} if $f(\sigma) \geq f(\sigma^{\prime})$ when $\sigma \succeq \sigma^{\prime}$, and that $\mu_N$ denotes the uniform measure on $\mathcal{S}_N$. The next lemma is due to Lacoin; see Proposition 3.6 in  \cite{Lacoin_2016}. 
 
 \begin{lemma}\label{lem:CensoringSkSpecial} Let $\nu_0$ be an initial distribution for the $S_2$ shuffle on $\mathcal{S}_N$ such that $\sigma \mapsto \frac{\nu_0}{\mu_N}(\sigma)$ is increasing.  Let $\mathcal{C}$ be a censoring scheme and let $\nu_t^{\mathcal{C}}$ be the law of the $S_2$ shuffle with respect to $\mathcal{C}$. Then $\sigma \mapsto \frac{\nu^{\mathcal{C}}_t}{\mu_N}(\sigma)$ is increasing and the censoring inequality holds.
 \end{lemma}
 
In the following, our goal to extend this result to general $k \geq 3$.

 \begin{lemma}\label{lem:CensoringSk}
  Let $k\geq 3$ and $\nu_0$ be an initial distribution for the $S_k$ shuffle on $\mathcal{S}_N$ such that $\sigma \mapsto \frac{\nu_0}{\mu_N}(\sigma)$ is increasing.  Let $\mathcal{C}$ be a censoring scheme and let $\nu_t^{\mathcal{C}}$ be the law of the $S_k$ shuffle with respect to $\mathcal{C}$. Then $\sigma \mapsto \frac{\nu_t^{\mathcal{C}}}{\mu_N}(\sigma)$ is increasing and the censoring inequality holds.
 \end{lemma}

We use the next lemma to approximate a single update in the  $\sk$ shuffle with censoring. 
 
 \begin{lemma}\label{lem:Approximation} 
   Let $\nu_t$ be the distribution of the $S_2$ shuffle on $\mathcal{S}_k$ at time $t$, then we have that
  \begin{equation}
  \lim_{t \rightarrow \infty}\TV{\nu_t-\mu_k} = 0 \, .
  \end{equation} 
  \end{lemma}
  \begin{proof} This is an immediate consequence of the fact that the $S_2$ shuffle is an irreducible continuous-time Markov chain, which has the uniform distribution as its unique stationary law.
  \end{proof}
  
 \begin{proof}[Proof of Lemma \ref{lem:CensoringSk}] We will in the following, only show that $\sigma \mapsto \frac{\nu_t^{\mathcal{C}}}{\mu_N}(\sigma)$ is increasing for any censoring scheme $\mathcal{C}$. The fact that the censoring inequality holds then follows from the same arguments as  Theorem 22.20 in \cite{LevinPeresWilmer}. To do so, we proceed by a proof by contradiction. Suppose there exists a censoring scheme $\mathcal{C}$, a sequence of times $(t^x_i)_{i\geq 1}^{x\in [N-1]}$, a time $t\geq 0$, some $\delta >0$, and permutations $\sigma \succeq \sigma^{\prime} $ such that 
\be \label{eq:contradiciton}
\P( \eta_t = \sigma^{\prime} \, | \,  \mathcal{T}^x_i=t_{i}^x ) - \P(\eta^{\mathcal{C}}_t =\sigma \, | \,  \mathcal{T}^x_i=t_{i}^x ) 
\geq \delta \eqpd
\ee
Let $M> 0$ which will be chosen later. Let $\mathcal{J}_{t} \coloneqq \{ (x,i) : t_x^i \leq t\}$, and let 
 $(\tilde{\eta}_t)_{t\geq 0}$ and $(\tilde{\eta}_t^{\mathcal{C}})_{t\geq 0}$ be two processes on $\mathcal{S}_N$ defined in the following way. For all $(x,i) \in \mathcal{J}_t$, at time $t_x^i$ in $(\tilde{\eta}_t)_{t\geq 0}$ we perform a sequence of $M$ many (discrete time) $S_2$ shuffle moves on the interval $[x, x+(k-1)]$. Similarly for the process $(\tilde{\eta}_t^{\mathcal{C}})_{t\geq 0}$, we apply for all $(x,i) \in \mathcal{J}_t$ a sequence of $M$ many $S_2$ shuffles, but for each interval in the decomposition $\mathcal{I}$ defined in \eqref{eq:Interval} for the censoring scheme at $t_x^i$ separately. By Lemma \ref{lem:Approximation} and a standard comparison between discrete time and continuous time Markov chains -- see Theorem 20.3 in \cite{Peres_Winkler_2013} -- and the triangle inequality for total-variation distance,  we can choose $M = M(t, \mathcal{J}_t, \delta, k, \mathcal{C})$ sufficiently large, such that
\begin{align}
\begin{split}
    \TV{\P( \tilde{\eta_t} \in \cdot \, | \,  \mathcal{T}^x_i=t_{i}^x ) - \P(\eta_t \in \cdot \, | \,  \mathcal{T}^x_i=t_{i}^x )}
&\leq \frac{\delta}{4} \\
 \TV{\P( \tilde{\eta_t}^{\mathcal{C}} \in \cdot \, | \,  \mathcal{T}^x_i=t_{i}^x ) - \P(\eta_t^{\mathcal{C}} \in \cdot \, | \,  \mathcal{T}^x_i=t_{i}^x )}
&\leq \frac{\delta}{4}\eqpd
\end{split}
\end{align}
Observe that $(\tilde{\eta}_t)_{t\geq 0}$, respectively $(\tilde{\eta}_t^{\mathcal{C}})_{t\geq 0}$, is an $S_2$ shuffles, respectively an $S_2$ shuffle with respect to some censoring scheme $\tilde{\mathcal{C}}$. Thus, using Lemma \ref{lem:CensoringSkSpecial}, and again the triangle inequality for total-variation distance we obtain the desired contradiction to \eqref{eq:contradiciton}.
\end{proof}

\begin{remark} \label{rem:Censor}
    Note that the same arguments as in the proof of Lemma \ref{lem:CensoringSk} apply to the $\sk$ shuffle with boundaries, establishing that the censoring inequality holds.
\end{remark}
\subsection{Preservation of the strong Rayleigh property}
In this section we discuss the strong Rayleigh property and its relation to negative dependence. Let $n \in\N$, and define a function $f \in \mathbb{C}\left[z_1, \ldots, z_n\right]$ with real coefficients to be \textbf{real stable}
if $f\left(z_1, \ldots, z_n\right) \neq 0$ whenever $\mathfrak{I m}\left(z_j\right)>0$ for $1 \leq j \leq n$.
Let $\pi$ be a probability measure over $\{0,1\}^n$. For $(X_1, \dots X_N)\sim \pi$ is called \textbf{strongly Rayleigh} if its generating polynomial 
\be
\left(z_1, \ldots, z_n\right) \longmapsto \mathbb{E}_\pi\left[\prod_{i=1}^n z_i^{X_i}\right]
\ee
is real stable. The strong Rayleigh property was introduced  by Borcea, Br\"{a}nd\'{e}n, and Liggett in \cite{borcea2009negative}. Recall for $K \in [N-1]$ the projection $(\xi_t^K)_{t\geq 0}$  of the $S_k$ shuffle to the first $K$ cards, defined in Section \ref{sec:skprojection}. The following lemma can be found as Proposition 5.1 in \cite{borcea2009negative}. 

\begin{lemma}[Proposition 5.1 in \cite{borcea2009negative}]
Let $K\in [N-1]$. Let $\nu_t$ denote the law of the projection of the $S_2$ shuffle to the first $K$ cards. If $\nu_0$ is strongly Rayleigh then so is the distribution of $\nu_t$ for all $t>0$.
\end{lemma}
We have the following simple consequence for the $\sk$ shuffle.
\begin{corollary}\label{cor:skstronglyrayleigh}
Let $K,k \in [N-1]$. Let $\nu_t$ denote the law of the projection of the $\sk$ shuffle to the first $K$ cards. If $\nu_0$ is strongly Rayleigh then so is the distribution of $\nu_t$ for all $t>0$. The same holds for the $\sk$ shuffle with boundaries and censoring.
\end{corollary}

\begin{proof}
The fact that any individual $S_{i+1}$ update of an interval $[x, x+i]$ for any $x \in [N-i]$ and $i \geq 1$ preserves the strong Rayleigh property is a consequence of Theorem 1.2 in~\cite{BorceaBranden2009}. Using the Trotter Product formula -- Theorem 3.44 in \cite{liggett2010continuous} --  we obtain the desired statement for the $\sk$ shuffle as the generator of the $\sk$ shuffle with boundaries and censoring can be written as the sum of generators of $S_i$ shuffle moves for time interval in which the censoring scheme remains constant.
\end{proof}
Next, we say that a set of random variables $\{X_1, \dots, X_n\}$ taking values in $\{0,1\}$ is \textbf{negatively dependent} if for all $S \subset [n]$, we have
\begin{equation}
    \bE\left[\prod\limits_{i\in S}X_i \right]\leq \prod\limits_{i \in S} \bE\left[X_i\right].
\end{equation}

In \cite{borcea2009negative} it is shown that strongly Rayleigh implies negative dependence, and we will use the following direct consequence of negative dependence, which we state without proof. 
\begin{corollary}\label{cor:varbound}
    Let $c_i \geq 0$ and let $Z_n \coloneqq  \sum_{i=1}^n c_i X_i$ be the sum of negatively dependent random variables $\{X_1, \dots X_n\}$ for some $n \in \N$. Then we have $$\Var[Z_n] \leq \sum_{i=1}^{n}c_i^2\Var[X_i]. $$ 
\end{corollary}

\section{Lower bounds on the mixing time of the $\sk$ shuffle}\label{sec:lowerbounds}

\subsection{An approximate second moment method}

For the $S_2$ shuffle sharp lower bounds can be obtained using Wilson's Lemma as first introduced in \cite{wilson2004mixing}, and approximate versions of his technique can be found in  \cite{gantert2020mixing} and  \cite{nam2019cutoff}. Here we rely instead on an approximate version of the second moment method originally introduced by Diaconis and Shashahani in~\cite{BL}. To state this approximate second moment method, consider a continuous-time Markov chain $\left(X_{t}\right)_{t \geq 0}$ with generator $\mathcal{A}$ on a finite state space $S$. It is a well known result that for any function $f: S \to \R$ the process $\left(M_{t}\right)_{t \geq 0}$ with
\be\label{condition}
M_{t}\coloneqq f\left(X_{t}\right)-f\left(X_{0}\right)-\int_{0}^{t}(\mathcal{A} f)\left(X_{s}\right) \mathrm{d} s \ \ \mbox{ for all } t \geq 0
\ee
is a martingale. We have the following result on the mixing time of $\left(X_{t}\right)_{t \geq 0}$.
\begin{lemma}\label{lem:gensecondmoment}
Let $\Psi: S \rightarrow \mathbb{R}$  be such that for some $\lambda \geq c>0$, and $R>0$ we have
\begin{equation}
\label{eq:diffbound}
|(-\mathcal{A} \Psi)(y)-\lambda \Psi(y)| \leq c \text { for all } y \in S,
\quad \text{ and } \quad
 \Var[\Psi(X_t)] \leq R \text { for all } t \geq 0\eqpd
\end{equation}
Then for all $\varepsilon \in(0,1)$, the mixing time $\tmix$ of $(X_t)_{t \geq 0}$ satisfies
\be\label{eq:mixingtimelowerbound}
\tmix(1-\varepsilon) \geq \frac{1}{\lambda} \log \left(\|\Psi\|_{\infty}\right)-\frac{1}{2 \lambda}\log\left(\frac{4\max(2R, c)}{\e}\right)\eqpd
\ee
\end{lemma}
\begin{proof}
Let $X_0=\eta$ almost surely for some $\eta \in S$ with $|\Psi(\eta)|=\|\Psi\|_{\infty}$. Let $\mu$ denote the stationary distribution of $\left(X_t\right)_{t \geq 0}$, and $X_{\infty} \sim \mu$. By \eqref{eq:diffbound} and the martingale  $(M_t)_{ t \geq 0}$, with $f\coloneqq\mathbb{E}\left[\Psi\left(X_t\right)\right]$ for all $t \geq 0$, we get
$$
f^{\prime}(t)=\mathbb{E}\left[(\mathcal{A} \Psi)\left(X_t\right)\right] \in[-\lambda f(t)-c,-\lambda f(t)+c] \text { for all } t \geq 0\eqpd
$$
Applying Gronwall's lemma yields
$$
f(t) \leq f(0) e^{-\lambda t}+\int_0^t c e^{-\lambda(t-s)} \mathrm{d} s \leq f(0) e^{-\lambda t}+\frac{c}{\lambda} \text { for all } t \geq 0 \eqcom
$$
and it follows that 
\be\label{eq:gronwallbound}
\left|f(t)-e^{-\lambda t} f(0)\right| \leq \frac{c}{\lambda} 
\ee
holds for all $t \geq 0$, by applying Gronwall's lemma to $-f$. Take $t$ equal to the right hand side of \eqref{eq:mixingtimelowerbound}. As a lower bound on the expectation of $\Psi$, we have
$$
\mathbb{E}\left[\Psi\left(X_t\right)\right] \geq e^{-\lambda t} \Psi\left(X_0\right)-\frac{c}{\lambda}=e^{-\lambda t}\|\Psi\|_{\infty}-\frac{c}{\lambda} \geq\frac{1}{2} e^{-\lambda t}\|\Psi\|_{\infty} \eqcom
$$
where the last inequality is due to the fact that we require $\lambda \geq c >0$, and our choice of $t$.

By taking $t \rightarrow \infty$ in \eqref{eq:gronwallbound}, we see that $\left|\mathbb{E}\left[\Psi\left(X_{\infty}\right)\right]\right| \leq c / \lambda$, using $\operatorname{Var}\left[\Psi\left(X_{\infty}\right)\right] \leq R$.
To bound the total-variation distance, letting  $P^\eta_t$ be the law of $X_t$ started from $\eta$, we get
\begin{align}
\begin{split}\label{eq:fancycheby}
\TV{ P^\eta_t - \mu } & \geq \mathbb{P}\left(\Psi\left(X_t\right) \geq \frac{1}{2} \mathbb{E}\left[\Psi\left(X_t\right)\right]\right)-\mathbb{P}\left(\Psi\left(X_{\infty}\right) \geq \frac{1}{2} \mathbb{E}\left[\Psi\left(X_t\right)\right]\right) \\
& \geq 1-4 \frac{\operatorname{Var}\left(\Psi\left(X_t\right)\right)}{\mathbb{E}\left[\Psi\left(X_t\right)\right]^2}-4 \frac{\operatorname{Var}\left(\Psi\left(X_{\infty}\right)\right)+\mathbb{E}\left[\Psi\left(X_{\infty}\right)\right]^2}{\mathbb{E}\left[\Psi\left(X_t\right)\right]^2} \eqpd
\end{split}
\end{align}
Here the last line follows from Chebyshev's inequality. By Markov's inequality
$$
\mathbb{P}\left(\Psi\left(X_{\infty}\right) \geq \frac{1}{2} \mathbb{E}\left[\Psi\left(X_t\right)\right]\right) \leq \mathbb{P}\left(\Psi\left(X_{\infty}\right)^2 \geq \frac{1}{4} \mathbb{E}\left[\Psi\left(X_t\right)\right]^2\right) \leq 4 \frac{\mathbb{E}\left[\Psi\left(X_{\infty}\right)^2\right]}{\mathbb{E}\left[\Psi\left(X_t\right)\right]^2} . 
$$
Substituting $t$ from \eqref{eq:fancycheby} yields the desired result.
\end{proof}

\subsection{A lower bound from the generalized second moment method}
In the following we prove a lower bound  on the mixing time for the $\sk$ shuffle with and without boundaries, which gives the lower bounds on the mixing time in Theorems \ref{thm:Precutoff} and \ref{thm:cutoff}. Recall that
\begin{align}\nonumber
    \Phi^{(j)}_N(\sigma)\coloneqq \sum_{x=1}^{N-1} h_{\sigma}^{\prime}(x)\psi_j(x) \quad \text{ where } \quad \psi_j(x) \coloneqq \sin \left(\frac{ x j\pi}{N}\right)
\end{align}
and, recalling the height function $(h_t^{\prime})_{t\geq 0}$ from Section \ref{sec:prelim}, we set
\be\label{eq:eigenfunctiontime}
\Phi_{N,t}^{(j)} \coloneqq  \sum_{x=1}^{N-1} h_{t}^{\prime}(x)\psi_j(x) 
\ee
for all $t\geq 0$. In the following, we use $ x \sim y$ to denote that $x$ is of order $y$.
\begin{lemma}\label{lem:wilsonlowerboundsk} Let $\mathcal{L}$ and $\lambda_{N,k}$ be as defined in \eqref{eq:skgenwithboundary} and \eqref{eq:eigenvaluebound}, and let $k = o(N^{3/4})$. Then for all $\sigma \in S_N$
\be\label{operatorbound2}
|(-\mathcal{L} \Phi_N^{(1)})(\sigma)-\lambda^{(1)}_{N,k} \Phi_N^{(1)}(\sigma)| \leq c
\ee
holds for $c \sim k^6\pi^3N^{-3}$, $R \sim N^3$,  and  $\|\Phi_N\|_{\infty} \sim N^2$. \end{lemma}
\begin{proof}
By Lemma \ref{lem:ApproxEigen} and Lemma \ref{lem:gensecondmoment}, we have that $c \sim k^6\pi^3N^{-3}$. Thus, it suffices to bound the variance for the approximate eigenfunction $\Phi_{N,t}^{(1)}$. Note that the initial distribution starting from the identity is strongly Rayleigh, and thus by Corollary~\ref{cor:skstronglyrayleigh} so is the distribution of the projection of the $\sk$ shuffle with boundaries on the first $N/2$ cards.
Let $(X_1^t, \dots X^t_N)$ be the projection of the $\sk$ shuffle with boundaries, where $X_i^t$ is the indicator function that the card at position $i$ has label at most $N/2$ at time~$t$. Then by Corollary \ref{cor:varbound}
\begin{align}
    \begin{split}
        \Var(\Phi_{N,t}^{(1)}) &= \Var\left(\sum_{x=1}^{N-1} h_{t}^{\prime}(x)\psi_j(x)\right) = \Var\left(\sum\limits_{m=1}^{N-1}\left(\sum\limits_{i=m+1}^{N}\psi_j(i)\right)X_m^t\right)\\
        &\leq \sum\limits_{m=1}^{N-1}\left(\sum\limits_{i=m+1}^{N}\psi_j(i)\right)^2\Var(X_m^t) \leq \sum\limits_{m=1}^{N-1}m^2\Var(X_m^t) \leq N^3
    \end{split}
\end{align}
allowing us to conclude. 
\end{proof}

\begin{proof}[Proof of the lower bounds in Theorems \ref{thm:Precutoff} and \ref{thm:cutoff}]
Combining Lemma \ref{lem:ApproxEigen} and \ref{lem:wilsonlowerboundsk} gives the desired lower bound on the mixing time for the $\sk$ shuffle with boundaries in Theorem~\ref{thm:cutoff}. To see that the corresponding lower bounds holds also for the $\sk$ shuffle without boundaries, note that the function 
\begin{equation}
    \sigma \mapsto \sum_{x=1}^{N-1} h_{\sigma}(x)\psi_1(x)
\end{equation} is increasing with respect to the partial order $\succeq$ defined in \eqref{def:PartialOrder}. Thus, the Lemma~\ref{lem:CensoringSk} and Remark \ref{rem:Censor}, treating the $\sk$ shuffle as an $\sk$ shuffle with boundaries and censoring
\begin{equation}
  \E\left[ \sum_{x=1}^{N-1} h_t(x)\psi_1(x) \right]   \geq  \E\left[ \sum_{x=1}^{N-1} h^{\prime}_t(x)\psi_1(x) \right] 
\end{equation} for all $t \geq 0$. The lower bound on the mixing times of the $\sk$ shuffle without boundaries follows from Chebyshev's inequality using Corollary \ref{cor:skstronglyrayleigh} and the same arguments as in Lemma \ref{lem:wilsonlowerboundsk} to bound the variance of the function $\Phi_{N,t}^{(1)}$ for the $\sk$ shuffle with boundaries.
\end{proof}

\begin{remark}\label{rem:lowerBound}
    Note that we in fact showed that the lower bound on the mixing time in Theorems \ref{thm:Precutoff} and \ref{thm:cutoff} remains valid for all $k=o(N^{3/4})$.
\end{remark}

\section{Upper bounds on the mixing time}\label{sec:generalupperbound}
\subsection{A general coupling for the $\sk$ shuffle}
\label{sec:monotonecoupling} In this section, we provide an upper bound on the mixing time of the $\sk$ shuffle. In contrast to our specific choice of boundary conditions in \eqref{eq:deltaWeights}, we allow in the following  for more general choices of the parameters $(\delta^{(k)}_i)$.  \\

We start by defining a coupling for the $\sk$ shuffle with boundaries. Let $(\zeta_t)_{t\geq0}$ and $(\zeta^{\prime}_t)_{t\geq0}$ denote the $\sk$ shuffles started from $\zeta,\zeta^{\prime} \in \mathcal{S}_N$, respectively. 
For both $\sk$ shuffles, we will use the same Poisson clocks, i.e. when we update an interval $[x,x+j]$ for some $x$ and $j$ in $(\zeta_t)_{t\geq0}$ at some time $s \geq 0$, we do the same in the process $(\zeta^{\prime}_t)_{t\geq0}$. Suppose that a clock associated with an interval $[x, x+j]$ rings at time $s$. Let $I_{x,s}\subseteq [N]$ be the set of labels for which both configurations agree at time $s$. For these $|I_{x,t}|$ cards, select $|I_{x,t}|$ of the $j+1$ positions in the interval $[x, x+j]$ uniformly at random, and assign the cards in both $\zeta_s$ and $\zeta^{\prime}_s$ whose labels are in $I_{x,t}$ to these positions. On the remaining $(j+1)-|I_{x,t}|$ positions, we distribute the cards in both configurations $\zeta_s$ and $\zeta^{\prime}_s$ uniformly at random and independently. \\

We refer to this as the \textbf{canonical coupling} for the $\sk$ shuffle, and write $\mathbf{P}$ for the joint law of $(\zeta_t)_{t\geq0}$ and $(\zeta^{\prime}_t)_{t\geq0}$ under this coupling. Let $Z_{i,t}$ and $Z^{\prime}_{i,t}$ denote the positions of the cards labeled $i$ in the configurations $\zeta_t$ and $\zeta_t^{\prime}$ respectively. Moreover, let $\tau_{i}$ be the first time at which the cards of label $i$ are located at the same position in both shuffles, and note that the cards of label $i$ occupy the same position for all $s \geq \tau_i$. The next proposition states an upper bound on the mixing time of the $\sk$ shuffle with and without boundaries.

\begin{proposition}
\label{prop:genupperbound}
Suppose that $\delta^{(k)}_i=\delta^{(k)}_{N-i}\in [0,1]$ for every  $i \in [k]$, and assume that $k=o(N^\frac{2}{3})$. Then there exists an absolute constant $C>0$ such that for all $\sigma, \sigma^{\prime} \in \mathcal{S}_N$, and all $t \geq CN^2k^{-3}\log(N)$, we have that for all $N$ sufficiently large
\begin{equation}\label{eq:MixingUpperProp}
\TV{\mathbf{P}(\zeta^{\prime}_t \in \cdot \, \mid \zeta^{\prime}_0 = \sigma^{\prime} )-\mathbf{P}(\zeta_t \in \cdot \, \mid \zeta_0 = \sigma )}\leq N^{-1} \eqpd
\end{equation}
For the $\sk$ shuffle with boundary rates $(\delta^{(k)}_i)$ from \eqref{eq:deltaWeights}, the upper bound on the total-variation distance in \eqref{eq:MixingUpperProp} continues to hold for all $k=o(N)$ and $t \geq CN^2k^{-3}\log(N)$.
\end{proposition}
The proof of Proposition \ref{prop:genupperbound} will be split in two main  parts. First, we investigate the time it takes for a single card to leave the sites $[4k]$, respectively $\{N-4k,\dots,N \}$, close to the boundary. In a second step, we establish tail estimates the coalescence time $\tau_i$ for cards of label $i$, and obtain the desired upper bound on the mixing time by a union bound.

\subsection{An estimate on the exit time from the boundary}

Consider in the following the positions $(Z_{1,t})_{t \geq 0}$ and $(Z^{\prime}_{1,t})_{t \geq 0}$ of the cards of label $1$. We denote by $P_{x}$, respectively $P^{\prime}_{x}$, the law of the cards of label $1$ when starting the cards from position $x\in [N]$ in $\zeta_0$ and $\zeta^{\prime}_0$, respectively. For all $y\in [N]$, let $\tilde{\tau}_{>y}$ and $\tilde{\tau}_{<y}$ 
be defined as
\begin{equation}
\tilde{\tau}_{>y} := \inf\{ t \geq 0 \colon Z_{1,t}>y \} \quad \text{ and } \quad \tilde{\tau}_{<y} := \inf\{ t \geq 0 \colon Z_{1,t}<y \}\eqcom
\end{equation} i.e.\ the first time at which the card of label $1$ in $(\zeta_t)_{t \geq 0}$ reaches a position larger than $y$, respectively smaller than $y$. For the following three lemmas, we assume that $k=o(N)$.
\begin{lemma}\label{lem:ExitBoundaryDeltaZero} Let $\delta^{(k)}_i=\delta^{(k)}_{N-i} \in [0,1]$ for every  $i \in [k]$. Then there exist absolute constants $c,C>0$ such that for all  $x \in [N]$, we have that for all $N$ sufficiently large
\begin{equation}
P_{x}\left(\tilde{\tau}_{>4k} > C \right) \leq c \eqpd
\end{equation}
\end{lemma}
\begin{proof}
Using the canonical coupling, we can assume without loss of generality that $x=1$. Since the interval $[k]$ is updated at rate $1$, and the cards are assigned to a position chosen uniformly at random, we see that $P_1(\tilde{\tau}_{>k/2} > 2) \leq  \frac{1}{4}$. Next, since all boundary rates are bounded by $1$ by our assumptions, note that the event that the first update involving card~$1$ after time $s_1$ is initiated by an interval $[j,j+k-1]$ for some $j > k/4$ has positive probability uniformly in $k$. Thus, we get that for some absolute constants $c_1,c_2>0$
\begin{equation}
    P_1(\tilde{\tau}_{>k/2}  > c_1 \, |  s_1 \leq  2) \leq c_2 \eqpd
\end{equation}
Iterating this argument for $\tilde{\tau}_{>mk/4}$ with $m\in [2,16]$, we conclude.
\end{proof}

We have a similar statement for the $\sk$ shuffle with boundary rates $(\delta_i^{(k)})$ from \eqref{eq:deltaWeights}. 
\begin{lemma}\label{lem:ExitBoundaryDeltaSuitable} Let $(\delta^{(k)}_i)$ be defined as in \eqref{eq:deltaWeights}. Then there exist absolute constants $c,C>0$ such that for all  $x \in [N]$, we have that for all $N$ sufficiently large
\begin{equation}
P_{x}\left(\tilde{\tau}_{>4k} > \frac{C}{k} \right) \leq c \eqpd
\end{equation}
\end{lemma}
\begin{proof} Using the canonical coupling, we can again assume without loss of generality that $x=1$. Note that until time $\tilde{\tau}_{>k/4}$, for any update of an interval $[j]$ for some $j \geq k/2$ before time $\tilde{\tau}_{>k/4}$, we have that card $1$ gets moved to some position $>k/4$ with probability at least $\frac{1}{4}$. Since $\delta^{(k)}_i \geq 1$ for all $i\in [k]$, and each interval $[j]$ is updated at rate at least $1$,
\begin{equation}
P_1\left( \tilde{\tau}_{>k/4} > \frac{4}{k} \right) \leq  \frac{1}{8}  \eqpd
\end{equation}
Since $\delta_{i}^{(k)}\leq 8$ for all $i \geq \frac{k}{4}$, note that the event that the first update involving card~$1$ after time $\tilde{\tau}_{>k/4}$ is initiated by an interval $[j,j+k-1]$ for some $j > k/8$ has positive probability, uniformly in $k$. Hence, we obtain that for some absolute constants $c_1, c_2>0$
\begin{equation}\label{eq:IterationEquation}
P_1\left( \tilde{\tau}_{>k/2} \geq \frac{c_1}{k} \, \Big| \, \tilde{\tau}_{>k/4} \leq  \frac{4}{k} \right) \geq  c_2 \eqpd
\end{equation} Iterating this argument for $\tilde{\tau}_{>mk/4}$ with $m\in [16]$, we conclude. 
\end{proof}

We require a final preliminary estimate on the expected return time for the $\sk$ shuffle with boundaries when the boundary rates $(\delta^{(k)}_i)$ are in $[0,1]$.

\begin{lemma}\label{lem:ExpectedTimeBoundary}
 Let $\delta^{(k)}_i=\delta^{(k)}_{N-i} \in [0,1]$ for every  $i \in [k]$. Then there exists an absolute constant $C>0$ such that for all $x\in [k,2k]$, and all $N$ sufficiently large
 \begin{equation}
 E_x[ \tilde{\tau}_{>4k} ] \leq \frac{C}{\sqrt{k}} \eqcom
 \end{equation} where $E_x$ denotes the expectation with respect to $P_x$.
\end{lemma}
\begin{proof}
From Lemma \ref{lem:ExitBoundaryDeltaZero}, we get that there exists some  $C_1>0$ such that for all $x\in [4k]$, we have that $E_x[ \tilde{\tau}_{>4k} ]\leq C_1$. Thus, using the canonical coupling to see that $E_x[ \tilde{\tau}_{>4k}]$ is decreasing in $x$, and iterating along $\tilde{\tau}_{>mk/4}$ for $m\in [16]$ as in \eqref{eq:IterationEquation},  it suffices to show that
\begin{equation*}
P_x\left( \tilde{\tau}_{<\sqrt{k}} <  \tilde{\tau}_{>4k} \right) \leq \frac{c_1}{\sqrt{k}} \quad \text{and} \quad E_{\sqrt{k}}\left[ \tilde{\tau}_{>k/4} \leq \frac{c_2}{\sqrt{k}} \, \Big| \,  Z_{1,t}\geq \sqrt{k} \text{ for all } t \in  [0,\tilde{\tau}_{>k/4}] \right] \leq \frac{c_3}{\sqrt{k}} 
\end{equation*} for all $x\in [k,2k]$ and constants $c_1,c_2,c_3>0$. The second inequality is immediate from the fact that on the event $\{ Z_{1,t}\geq \sqrt{k} \text{ for all } t \geq 0  \}$, with positive probability $\tilde{\tau}_{>k/4} \leq c_4k^{-1/2}$ holds for some constant $c_4>0$ by considering the first time an interval $[j,j+k-1]$ for $j<\sqrt{k}$ is updated. To see the first inequality, note that for each update of interval  containing card $1$, we have a positive probability, uniformly in $k$ and the position $x \in [\sqrt{k},k]$ of card $1$, that card $1$ is moved to some position $>k/4$, while the probability to move card $1$ to the first $\sqrt{k}$ positions is at most $c_5k^{-1/2}$ for some absolute constant $c_5>0$. 
\end{proof}

\subsection{Proof of the upper bound on the mixing time}
We start by showing that with positive probability, the time for card $1$ to exceed  $\lfloor \frac{N}{2} \rfloor$ is of order at most $N^{2}k^{-3}$.

\begin{lemma}\label{lem:MartingaleInterior} Suppose that $\delta^{(k)}_i=\delta^{(k)}_{N-i}\in [0,1]$ for every  $i \in [k]$,  and $k=o(N^\frac{2}{3})$. Then for all $x\in [N]$, we have that for some positive constants $c_1,c_2>0$, and $N$ sufficiently large,
\begin{equation}\label{eq:HitMidZero}
P_{x}\left( \tilde{\tau}_{>\lfloor N/2\rfloor} > c_1N^2k^{-3}  \right) \leq c_2 \, .
\end{equation} 
For the $\sk$ shuffle with rates $(\delta^{(k)}_i)$ from \eqref{eq:deltaWeights},  \eqref{eq:HitMidZero} continues to hold for all $k=o(N)$.
\end{lemma}
\begin{proof} In the following, we define a stopping time  $\tau_{\text{hit}}$ for the process $(Z_{1,t})_{t \geq 0}$ by
\begin{equation}
\tau_{\text{hit}} := \min\left( \tilde{\tau}_{<2k} , \tilde{\tau}_{>N-2k}  \right) \eqpd
\end{equation} Note that $(Z_{1,t})_{t \in [0,\tau_{\text{hit}}]}$ is a stopped martingale, and by the optional stopping theorem
\begin{equation}\label{eq:Gamblers}
P_{4k}\left(  \tilde{\tau}_{<2k} < \tilde{\tau}_{>\lfloor N/2 \rfloor}  \right) = \frac{2k}{\lfloor N/2 \rfloor-2k} \eqpd
\end{equation} Let $X$ be the amount of time $(Z_{1,t})_{t \geq 0}$ spends until time $\tilde{\tau}_{>4k}$ at sites $[2k]$. From Lemma~\ref{lem:ExitBoundaryDeltaSuitable} and Lemma~\ref{lem:ExpectedTimeBoundary}, we obtain that for some constant $c>0$, and all $x\in [4k]$
\begin{equation}\label{eq:BoundaryTime}
E_x[X] \leq \begin{cases} ck^{-1/2}& \text{ if } \delta^{(k)}_i\in [0,1] \text{ for all } i\in [k] \\
ck^{-1}& \text{ for rates } (\delta^{(k)}_i) \text{ from \eqref{eq:deltaWeights}} \eqpd 
\end{cases}
\end{equation}
For the stopped martingale $(Z_{1,t})_{t \in [0,\tau_{\text{hit}}]}$, note that its quadratic variation $(\langle Z_{1,\cdot}\rangle_t)_{t \in [0,\tau_{\text{hit}}]}$ satisfies for some constant $c^{\prime}>0$, and all $t\geq 0$
\begin{equation}
\langle Z_{1,\cdot}\rangle_t \geq c^{\prime} k^3 t
\end{equation} as card $1$ moves at rate $k$ according to an increment with    a variance of order $k^2$. By the optional stopping theorem for the martingale $(M_t)_{t \in [0,\tau_{\text{hit}}]}$ with  $M_t=(Z_{1,t})^2 - \langle Z_{1,\cdot}\rangle_t$, we see that $E_{4k}[\tau_{\text{hit}}]$ is of order $N^{2}k^{-3}$. Together with \eqref{eq:Gamblers} and \eqref{eq:BoundaryTime}, we conclude.
\end{proof}


\begin{proof}[Proof of Proposition \ref{prop:genupperbound}]
In the following, we argue that there exist $c_1,c_2>0$ such that
\begin{equation}\label{eq:IndividualCards}
\mathbf{P}( \tau_1 > c_1 N^2 k^{-3}) \leq c_2 \eqpd
\end{equation} for any pair of starting configurations $\zeta_0,\zeta^{\prime}_0$. Since the canonical coupling preserves the coalescence of cards, the upper bound on the mixing time then follows from iterating  \eqref{eq:IndividualCards} for all cards with labels in $[N]$, and a union bound. By Lemma \ref{lem:MartingaleInterior}, for some $c_3,c_4>0$\begin{equation}
\mathbf{P}( \tilde{\tau}_{>\lfloor N/2\rfloor}>c_3 N^2 k^{-3}) \geq \c_4\eqpd
\end{equation}  Since $(Z_{1,t})_{t \in [0,\tilde{\tau}_{<N/4} \wedge \tilde{\tau}_{>3N/4}]}$ is a stopped martingale with increments bounded by $k$, 
\begin{equation}\label{eq:StayPut}
\mathbf{P}\left( Z_{1,t} \in \left[ \frac{N}{4}, \frac{3N}{4} \right] \text{ for all } t\in \left[ \tilde{\tau}_{>\lfloor N/2\rfloor}, \tilde{\tau}_{>\lfloor N/2\rfloor} + CN^2k^{-3} \right] \right) > c_5 
\end{equation} for all $C>0$ and some $c_5=c_5(C)>0$. Conditioning on the event in \eqref{eq:StayPut}, note that by Lemma \ref{lem:MartingaleInterior} we can choose $C>0$ such that with positive probability, there exists some $t_{\ast} \in [ \tilde{\tau}_{>\lfloor N/2\rfloor}, \tilde{\tau}_{>\lfloor N/2\rfloor} + CN^2k^{-3}]$ such that $N/5< Z_{1,t_{\ast}}^{\prime}< 4N/5$. Using the Strong Markov property under the coupling $\mathbf{P}$, we see that  $(Z_{1,t})_{t  \geq t_{\ast}}$ and $(Z_{1,t}^{\prime})_{t  \geq t_{\ast}}$ coalesce with positive probability before hitting $2k$ or $N-2k$. This gives \eqref{eq:IndividualCards}, and hence finishes the proof.
\end{proof}

\section{Cutoff for  the $\sk$ shuffle with boundaries}\label{sec:lacoinupperbound}

\subsection{Approximate Fourier Analysis}
Recall from Section \ref{sec:prelim} the height function $(h_{\zeta_t})_{t \geq 0}$ of the $\sk$ shuffle with boundaries $(\zeta_t)_{t\geq 0}$ and, with a slight abuse of notation, set
\be
h_t^{\prime}(x,y) \coloneqq h_{\zeta_t}(x,y) \eqpd
\ee
In \cite{Lacoin_2016}, a key observation is that the expected height function $(x,y,t) \mapsto \bE[{h^{\prime}_t(x,y)}]$ of the $S_2$ shuffle is a solution $f: \{0, \dots, N\}^2 \times \R_0^{+} \to \R$ to the discrete heat equation 
\be
\begin{cases}
\partial_t f = \Delta_x f\\
f(0,y,t) = f(N,y,t) = 0\\
f(x,y,0) = \bE[{h^{\prime}_0(x,y)}] 
\end{cases}
\ee 
where $\Delta_x$ denotes the discrete Laplace operator 
\be
\Delta_x(f)(x) \coloneqq \frac{1}{2}f(x-1, y ,t) + \frac{1}{2}f(x+1, y, t) -f(x) \eqpd
\ee
This allows for sharp estimates on $\bE[{h^{\prime}_t(x,y)}]$ for the $S_2$ shuffle; see Lemma 4.1 in~\cite{Lacoin_2016}. Then next lemma provides a similar result for the $\sk$ shuffle with boundaries for $k \geq 3$.
\begin{lemma}\label{lem:upperbexpectedheight}
Let $k \geq 3$. There exists a constant $C>0$ such that for any initial configuration $\sigma \in \mathcal{S}_N$ of the $\sk$ shuffle with boundaries, for all $t\geq 0$, and all $y\in [N]$, 
\be\label{eq:MainFourierStatement}
\max\limits_{x \in \{0, \dots, N\}}\bE[h_t^{\prime}(x,y)] \leq 8\min(y, N-y)e^{-t\cdot\lambda_{1,N,k}} + Ck^3.
\ee
\end{lemma}
\begin{proof}
By the standard Fourier decomposition, we obtain that
\be\label{eq:FourierExpectedHeight}
\bE[h_t^{\prime}(x,y)] = \frac{2}{N}\sum\limits_{i=1}^N \sin\left( \frac{i \pi x}{N}\right)\sum\limits_{j=1}^N \bE\left[h_t^{\prime}(j,y)\sin\left( \frac{ij \pi}{N}\right)\right].
\ee
Recall $\Phi_{N,t}^{(j)}$ from \eqref{eq:eigenfunctiontime} and the  approximate eigenvalues $\lambda_{j,N,k}$ from \eqref{eq:eigenvaluebound}. By Lemma \ref{lem:ApproxEigen} and the same arguments as for \eqref{eq:gronwallbound} in Lemma \ref{lem:gensecondmoment}, we see that for all $y \in [N-1]$ and $t\geq 0$
\be
\left| \bE\left[\Phi_{N,t}^{(i)}\right] - e^{-t \cdot \lambda_{i,N,k}}\Phi_{N,0}^{(i)}(y)\right|\leq \frac{c_i}{\lambda_{i,N,k}} \eqcom
\ee
where we set $c_i = Ci^3k^6N^{-3}$ and take $C$ from Lemma \ref{lem:ApproxEigen}. 
Together with \eqref{eq:FourierExpectedHeight} we have
\be
\begin{split}\label{expectedheightupperb}
\abs{\bE[h^{\prime}_t(x,y)] } \leq \left|\frac{2}{N}\sum\limits_{i=1}^{N} \sin\left( \frac{i \pi x}{N}\right)e^{-t \cdot \lambda_{i,N,k}}\Phi_{N,0}^{(i)}(y) \right|+\left| \frac{2}{N}\sum\limits_{j=1}^N \left|\sin\left( \frac{j \pi x}{N}\right)\right| \frac{c_j}{\lambda_{j,N,k}} \right|\eqpd
\end{split}
\ee
Note that the second summand in \eqref{expectedheightupperb} is bounded from above by $2Ck^3$. For the first summand, note that $\abs{\Phi_{N}^{(i)}(y)}\leq \min(y,N-y)N$ for all $y\in [N]$ and $i \in [N-1]$. Moreover,  $\lambda_{j,N,k} \geq j \cdot \lambda_{1,N,k}$ holds for all $j \geq 1$. Using the fact that $\left|\sin(z)\right|<|z|$ for all $z \in \R$, we get  
\be\label{expectedheightfinalb}
\left|\bE[h^{\prime}_t(x,y)]\right| \leq 2Ck^3+ 8y\sum\limits_{j=1}^{N-1} e^{-jt\lambda_{1,N,k}} \leq 2Ck^3+\frac{8ye^{-t\lambda_{1,N,k}}}{1-e^{-t\lambda_{1,N,k}}} \eqpd
\ee
When $e^{-t\lambda_{1,N,k}} \leq \frac{1}{2}$, this allows us to conclude \eqref{eq:MainFourierStatement}. For $e^{-t\lambda_{1,N,k}} > \frac{1}{2}$,  \eqref{eq:MainFourierStatement}  is immediate from the fact that $\max_{x \in \{0,\dots,N\}}h_{\sigma}(x,y) \leq \min(y,N-y)$ for all $y\in [N]$ and $\sigma \in \mathcal{S}_N$.
\end{proof}

\subsection{Proof of the upper bound in Theorem \ref{thm:cutoff}} As we follow in large parts the arguments of Lacoin in \cite{Lacoin_2016}, we give the necessary adjustments, rather than  the arguments in full detail. 
Let us start by introducing the main objects and outlining the main strategy for the proof. Fix some $K\in \N$ chosen later, and recall the height function representation $h_{\sigma}(x,y)$ from \eqref{eq:heightfunction} for a permutation $\sigma \in \mathcal{S}_N$. Moreover, we recall the following definitions from \cite{Lacoin_2016}. \\
 
Let $x_i \coloneqq \lceil iN/K\rceil$ for all $i \geq 0$. For a permutation $\sigma \in \mathcal{S}_N$, we define in the following two projections $\widehat{\sigma}$ and $\bar{\sigma}$. The \textbf{semi-skeleton projection} $\widehat{\sigma}=(\widehat{\sigma}_{x,i})_{x\in [N],i\in [K]}$ and \textbf{skeleton projection} $\bar{\sigma}=(\bar{\sigma}(m,i))_{i,m \in [K]}$ are given by
\begin{equation}
  \widehat{\sigma}(x,i) = h_\sigma(x,x_i)\quad \text{ and } \quad  \bar{\sigma}(m,i) =h_{\sigma}(x_m, x_i)   
\end{equation}
 respectively. We denote by $\widehat{\mathcal{S}}_N$ and $\bar{\mathcal{S}}_N$ the corresponding image spaces of $\mathcal{S}_N$ under these  projections. Given a probability measure $\nu$ on $\mathcal{S}_N$, we use $\widehat{\nu}$, respectively $\bar{\nu}$, to denote the image measures of $\nu$ on $\widehat{\mathcal{S}}_N$, respectively $\bar{\mathcal{S}}_N$, under the semi-skeleton and skeleton projection, i.e.\ we set for all $\widehat{\sigma} \in \widehat{\mathcal{S}}_N$ and $\bar{\sigma} \in \bar{\mathcal{S}}_N$
 \be
  \widehat{\nu}(\widehat{\sigma}) \coloneqq \nu(\{\sigma  \in \mathcal{S}_N : \sigma \mapsto \widehat{\sigma}\}) \quad \text{ and } \quad 
\bar{\nu}(\bar{\sigma}) \coloneqq \nu(\{\sigma  \in \mathcal{S}_N : \sigma \mapsto \bar{\sigma}\}) \eqpd
\ee 
Let $\Delta x_i \coloneqq x_{i} -x_{i-1}$, and let $\widetilde{\mathcal{S}}_N$ be the largest subgroup of $\mathcal{S}_N$ which is for all $i\in [K]$ invariant under permuting the cards of labels between $x_{i-1}+1$  and $x_i$. Note that $\widetilde{\mathcal{S}}_N$ is isomorphic to the product space $ \bigotimes_{i=1}^K\mathcal{S}_{\Delta x_i}$. For a probability measure $\nu$ on $\mathcal{S}_N$, we define the measure $\tilde{\nu}$ on $\mathcal{S}_N$ by setting for all $\sigma \in \mathcal{S}_N$
\be\label{eq:forceduniformmeasure}
\widetilde{\nu}(\sigma) \coloneqq \frac{1}{\Pi_{i=1}^K(\Delta x_i)}\sum_{\widetilde{\sigma} \in \widetilde{\mathcal{S}}_N}\nu (\widetilde{\sigma} \circ \sigma) \eqpd
\ee In words, to obtain the measure $\widetilde{\nu}$ from $\nu$, we apply a uniformly chosen permutation which only shuffles for all $i\in [K]$ the cards of labels $x_{i-1}+1$ to $x_i$ among each other. \\

To show to upper bound on the mixing time in Theorem \ref{thm:cutoff}, let $\delta>0$ and set
\begin{equation*}
    t_1=  \frac{\delta N^2}{3 k^3}\log(N) \qquad  t_2= \left(\frac{2\delta}{3}+ \frac{4}{\pi^2}\right) \frac{N^2}{k^3}\log(N) \qquad  t_3= \left(\delta+ \frac{4}{\pi^2}\right) \frac{N^2}{k^3}\log(N) \eqpd
\end{equation*} We consider the censoring scheme $\mathcal{C}=(\mathcal{C}_t)_{t \geq 0}$ for the $\sk$ shuffle with boundaries given by
\begin{equation}\label{eq:CensoringCutoff}
    \mathcal{C}_t := \begin{cases}
    \{ \{x_i,x_i+1\} \colon i\in [K]\}    & \text{ if } t\in [0, t_1) \cup [t_2, t_3) \\
     \emptyset   & \text{ if } t\notin [0, t_1) \cup [t_2, t_3) \eqpd
    \end{cases}
\end{equation} 
In the following, let $(\nu_t)_{t \geq 0}$ be the law of the $\sk$ shuffle with  boundaries under the censoring scheme $\mathcal{C}$ from  \eqref{eq:CensoringCutoff} started from the identity, and let $\mu=\mu_N$ denote the uniform distribution on $\mathcal{S}_N$.
The proof proceeds now in two steps. First, we argue that by time $t_1$, for all $i\in [K]$, the cards of labels $x_{i-1}+1$ to $x_i$ in the $\sk$ shuffle with boundaries and censoring scheme $\mathcal{C}$ are well mixed among each other. Second, we argue that by time $t_3$, the semi-skeleton has well mixed. To do so, the key task is to verify that the skeleton of the $\sk$ shuffle with boundaries and censoring scheme $\mathcal{C}$ has well mixed by time $t_2$. 
This strategy is summarized and made precise in the following two propositions. 
\begin{proposition}[Proposition 5.1 in \cite{Lacoin_2016}]\label{prop:upperboundt1}
 For all $\e>0$, there exists some $N_0=N_0(\varepsilon)$ such that for all $N \geq N_0$ and $t \geq t_1$
\begin{equation}
\TV{\widetilde{\nu}_t - \nu_t} \leq \e/3 \eqpd
\end{equation}
\end{proposition}

The proof of Proposition \ref{prop:upperboundt1} is deferred to Section \ref{sec:FirstProp}.

 \begin{proposition}[Proposition 5.3 in \cite{Lacoin_2016}]\label{prop:upperboundt3}
For all $\e>0$ and $k=o(N^{1/6})$, there exists some $N_1=N_1(\varepsilon)$ such that for all $N \geq N_1$ and $t \geq t_3$
\begin{equation}
\TV{\widehat{\nu}_t - \widehat{\mu}} \leq 2\e/3 \eqpd
\end{equation}
\end{proposition}

The proof of Proposition \ref{prop:upperboundt3} is deferred to Section \ref{sec:SecondProp}.

\begin{proof}[Proof of the upper bound in Theorem \ref{thm:cutoff}]
Note that since the $\sk$ shuffle with boundaries $(\zeta_t)_{t \geq 0}$ is a transitive Markov chain -- see Section 2.6.2 in \cite{LevinPeresWilmer} -- it suffices to bound the distance from the stationary distribution when starting from the identity. Since the Dirac measure on the identity is increasing with respect to the partial order $\succeq$ from \eqref{def:PartialOrder}, using the censoring inequality Lemma~\ref{lem:CensoringSk} for the first step, and Lemma~4.3 in \cite{Lacoin_2016} for the second step, we obtain that
\begin{equation}
 \TV{ \P( \zeta_{t_3} \in \cdot \, | \, \zeta_{0}=\textup{id} ) - \mu } \leq  \TV{ {\nu}_{t_3} - \mu} \leq \TV{ \widehat{\nu}_{t_3} - \widehat{\mu}} + \TV{  \widetilde{\nu}_{t_3} - {\nu}_{t_3}} \eqpd
\end{equation} As $\delta>0$ for $t_3$ was arbitrary, we combine Propositions \ref{prop:upperboundt1} and  \ref{prop:upperboundt3} to conclude.
\end{proof}

\subsection{Proof of Proposition \ref{prop:upperboundt1}}\label{sec:FirstProp}

In the following, we will only describe the necessary changes in the proof of Proposition 5.1 in \cite{Lacoin_2016} in order to obtain Proposition \ref{prop:upperboundt1} for the $\sk$ shuffle with boundaries, and refer to Section 5.3 of \cite{Lacoin_2016} for a detailed proof of the corresponding result for the $S_2$ shuffle. Note that as we start from the Dirac measure on the identity $\delta_{\textup{id}}$, the measure $\widetilde{\delta_{\textup{id}}}$ can be identified with the product measure $\bigotimes_{i=1}^{K} \mu_{\Delta x_i}$ on $ \bigotimes_{i=1}^K\mathcal{S}_{\Delta x_i}$, where we recall \eqref{eq:forceduniformmeasure} and that $\mu_i$ denotes the uniform distribution on $\mathcal{S}_i$. By our choice of the censoring scheme $\mathcal{C}$, note that the measure $\nu_t$ for $t < t_1$ corresponds to the law of $K$ many $\sk$ shuffles on $\bigotimes_{i=1}^K\mathcal{S}_{\Delta x_i}$ with laws $(\nu^{i}_t)_{t \geq 0}$ and suitable boundary parameters $(\delta_i^{(k)})$, satisfying the assumptions of Proposition \ref{prop:genupperbound}, i.e.\ we have $\delta_i^{(k)}\in [0,1]$ for all $\mathcal{S}_{\Delta x_j}$ with $j\in \{ 2,\dots,K-1\}$. Using the canonical coupling from Section \ref{sec:generalupperbound} and Proposition \ref{prop:genupperbound}, we choose $K=K(\delta)$ large enough such that for all $N$ sufficiently large
\begin{equation}
\TV{ \nu_{t} - \widetilde{\nu}_{t_1} }  \leq \sum\limits_{i =1 }^{K} \TV{ \nu_{t}^i - \mu_i } \leq \sum\limits_{i =1 }^{K} \TV{ \nu_{t_1}^i - \mu_i } \leq  \frac{\e}{3} \eqpd
\end{equation}

\subsection{Proof of Proposition \ref{prop:upperboundt3} }\label{sec:SecondProp}

In order to show Proposition \ref{prop:upperboundt3}, we first state a bound on the skeleton projection $\bar{\nu_t}$. 

\begin{proposition}[Proposition 5.2 in \cite{Lacoin_2016}]\label{prop:upperboundt2}
For all $\e>0$ and $k=o(N^{1/6})$, there exists some $N_2=N_2(\varepsilon)$ such that for all $N\geq N_2$ and  $t \geq t_2$ 
\be\label{eq:StatementUpper3}
\TV{  \bar{\nu}_t -\bar{\mu}_t } \leq \e/3 \eqpd
\ee
\end{proposition}
\begin{proof} Let $\nu$ be a probability measure on $\mathcal{S}_N$, which is increasing with respect to the partial order $\succeq$. 
Lemma 5.5 in \cite{Lacoin_2016} states that  \eqref{eq:StatementUpper3} for some $\varepsilon>0$ follows whenever for some sufficiently small $\kappa(\varepsilon,K)$, one can show that
\begin{equation}
\bE_{\nu}\left[\sum\limits_{i=1}^{K-1}\sum\limits_{j=i}^{K-1} h_{\sigma}(x_i,x_j)\right] < \kappa \sqrt{N} \eqcom 
\end{equation} Here $\bE_{\nu}$ denotes the expectation with respect to $\nu$, and we let $\sigma \sim \nu$. Choose $\nu=\nu_{t_2}$ and note that $\nu_{t_2}$ is increasing by Lemma 
~\ref{lem:CensoringSk}. As $k=o(N^{1/6})$, we get from Lemma \ref{lem:upperbexpectedheight} that starting the $\sk$ shuffle with boundaries from $\nu_{t_1}$, for any $\kappa>0$ and all $N$ sufficiently large
\be
\begin{split}
\bE_{\nu}\left[\sum_{i=1}^{K-1}\sum_{j=i}^{K-1} h_{\sigma}(x_i,x_j)\right]&\leq (K-1)^2(2Ne^{-(t_2-t_1)\lambda_{N,k}} +ck^3)) \leq \kappa \sqrt{N} 
\end{split} \eqcom
\ee  allowing us to conclude. 
\end{proof}

It remains to deduce Proposition \ref{prop:upperboundt3} from Proposition \ref{prop:upperboundt2}. As this follows along the same lines as the proof of Proposition 5.3 in \cite{Lacoin_2016}, we will highlight only the required adjustments. 
Let $\sigma_{t} \sim \nu_{t}$ be the configuration of the $\sk$ shuffle at time $t$, and observe that the skeleton $\bar{\sigma}_{t_2}$ remains unchanged for times $t \in [t_2,t_3]$ by our choice of the censoring scheme~$\mathcal{C}$. Conditioning on the event $\{\sigma_{t_2}=\xi \}$ for some $\xi \in \mathcal{S}_N$, let $\mu_{\bar{\xi}}$ be the uniform distribution on set of permutations with skeleton $\bar{\xi}$. From the definition of the semi-skeleton for the first step, and Proposition \ref{prop:genupperbound} together with the same reasoning as in Proposition \ref{prop:upperboundt1} of decomposing the $\sk$ shuffle into $K$ independent $S_{\Delta x_i}$ shuffles for $i\in [K]$ in  the second step, 
\begin{equation}\label{eq:UpperHelper}
    \TV{\hat{\nu}_{t_3}( \, \cdot \, | \bar{\sigma}_{t_2}=\bar{\zeta}) - \hat{\mu}( \, \cdot \, | \bar{\sigma}_{t_2}=\bar{\zeta} ) } \leq \max_{\xi \in \mathcal{S}_N } \TV{ \P( \sigma_{t_3} \in \cdot \, | \, \sigma_{t_2}=\xi ) - \mu_{\bar{\xi}} } \leq \frac{\varepsilon}{3}
\end{equation} for all $\zeta \in \mathcal{S}_N$ and $N$ sufficiently large; see also equation (5.38) in \cite{Lacoin_2016}. Since we have that
\begin{equation}
    2 \TV{\hat{\nu}_{t_3}- \hat{\mu} } \leq 2 \left( \TV{  \bar{\nu}_t -\bar{\mu}_t } + \sum_{\bar{\zeta} \in \bar{\mathcal{S}}_N } \bar{\nu}_{t_3}(\bar{\zeta}) \TV{\hat{\nu}_{t_3}( \, \cdot \, | \bar{\sigma}_{t_2}=\bar{\zeta}) - \hat{\mu}( \, \cdot \, | \bar{\sigma}_{t_2}=\bar{\zeta} ) } \right) 
\end{equation} by equation (5.39) in \cite{Lacoin_2016}, we get Proposition \ref{prop:upperboundt3} by  combining Proposition \ref{prop:upperboundt2} and \eqref{eq:UpperHelper}.

\section{Comparison between the $\sk$ shuffle with and without boundaries}\label{sec:Comparison}

In Theorem \ref{thm:Precutoff}, we saw for $k=o(N^{2/3})$ that the $\sk$ shuffle exhibits \textbf{pre-cutoff}, i.e.\ the ratio between the $\varepsilon$-mixing time and $N^{2}k^{-3}\log(N)$ is bounded for $N \rightarrow \infty$ from below and above by positive constants, which do not depend on $\varepsilon \in (0,1)$. While the $\sk$ shuffle with boundaries also exhibits pre-cutoff when  $k=o(N^{2/3})$, we argue that for $k=N^{\delta}$ with $\delta \in (\frac{2}{3},\frac{3}{4})$ the behavior of the $\sk$ shuffle and the $\sk$ shuffle with boundaries is fundamentally different due to a different treatment of the cards near the boundaries. 

\begin{proposition}\label{pro:Comparison}
For all $k=o(N^{3/4})$ and $\varepsilon \in (0,1)$, the mixing time $\tmix^{\prime}(\varepsilon)$ of the $\sk$ shuffle with boundaries satisfies
\begin{equation}\label{eq:LargeKWith}
\frac{6}{\pi^2} \leq \liminf_{N \rightarrow \infty} \frac{k(k^2-1)\cdot \tmix^{\prime}(\varepsilon)}{N^{2}\log(N)} \leq \limsup_{N \rightarrow \infty} \frac{k(k^2-1)\cdot \tmix^{\prime}(\varepsilon)}{N^{2}\log(N)}\leq c
\end{equation} for some constant $c>0$. In particular, pre-cutoff occurs. 
For $k=N^{\delta}$ with $\delta \in (\frac{2}{3},1)$, the mixing time $\tmix(\varepsilon)$ of the $S_k$ shuffle is of constant order where 
 \begin{equation}\label{eq:LargeKWithout}
     \lim_{\varepsilon \rightarrow 0} \liminf_{N \rightarrow \infty }\tmix(\varepsilon) = \infty \eqcom
 \end{equation} and for any fixed $\varepsilon>0$, there exist some $C=C(\varepsilon)>0$ such that \begin{equation}\label{eq:UpperSlow}
 \limsup_{N \rightarrow \infty } \tmix(\varepsilon) \leq C \eqpd
 \end{equation} In particular, pre-cutoff does not occur. 
\end{proposition}

In order to show Proposition \ref{pro:Comparison}, we require some setup. Set $\delta^{\prime}=\frac{1}{24}(3\delta-2)>0$, and recall from Section \ref{sec:generalupperbound} that we denote by $(Z_{i,t})_{t \geq 0}$ and $(Z^{\prime}_{i,t})_{t \geq 0}$ the position of the cards of label~$i$ in two $\sk$ shuffles under the canonical coupling $\mathbf{P}$. For $i\in [N]$ and $T \geq 0$, we define
\begin{equation}
B_{i,T} := \left\{ Z_{i,t} \in \big[N^{\frac{1}{3}+\delta^{\prime}},N-N^{\frac{1}{3}+\delta^{\prime}}\big] \text{ for all } t \in \big[T,T+N^{-\delta^{\prime}}\big]  \right\}  \eqcom
\end{equation} and similarly define $B_{i,T}^{\prime}$ with respect to $(Z^{\prime}_{i,t})_{t \geq 0}$.
We have the following result on the coalescence time $\tau_i$ of the cards of label $i \in [N]$ under the canonical coupling $\mathbf{P}$.
\begin{lemma}\label{lem:CoalescenceCompare} Let $\delta > \frac{2}{3}$ and  $i\in [N]$. Then for all $T\geq 0$ and $N$ sufficiently large
\begin{equation}\label{eq:AssumptionBulk}
\mathbf{P}\big(B_{i,T}\, \big|\, Z_{i,T} \in \big[N^{\frac{1}{3}+\delta^{\prime}},N-N^{\frac{1}{3}+\delta^{\prime}}\big] \big) \geq 1- N^{-3\delta^{\prime}} \eqcom
\end{equation}  and similarly for the events $B_{i,T}^{\prime}$. Moreover, we have that for all $N$ sufficiently large
\begin{equation}\label{eq:QuickMix}
 \mathbf{P}\big( \tau_{i}> T+N^{-\delta^{\prime}} \,  \big| \, B_{i,T} \cap  B_{i,T}^{\prime}\big) \leq N^{-2} \eqpd
\end{equation} 
\end{lemma}
\begin{proof} For the first statement \eqref{eq:AssumptionBulk}, note that $Z_{i,t}<N^{\frac{1}{3}+\delta^{\prime}}$ for some $t \geq T$ can only occur by an update of an interval $[j,j+k-1]$ for some $j<N^{\frac{1}{3}+\delta^{\prime}}$. Let $X_T$ be the total number of updates of these intervals between time $T$ and $T+N^{-\delta^{\prime}}$, and note that each such update places the card of label $i$ in the first $N^{\frac{1}{3}+\delta^{\prime}}$ positions independently with probability at most~$N^{\frac{1}{3}+\delta^{\prime}-\delta}$. As $X_T$ is Poisson-$(N^{1/3})$-distributed, we have that $\mathbf{P}(X_T \geq N^{\frac{1}{3}+\delta^{\prime}})\leq \frac{1}{4}N^{-3\delta^{\prime}}$ uniformly in $T>0$, and for all $N$ large enough. Hence
\begin{align*}
\mathbf{P}\big( Z_{i,T} > N^{\frac{1}{3}+\delta^{\prime}} &\text{ for all } t\in [T,T+N^{-\delta^{\prime}}] \, \big| \, Z_{i,T} \in \big[N^{\frac{1}{3}+\delta^{\prime}},N-N^{\frac{1}{3}+\delta^{\prime}} \big]\big) \\ &\geq \mathbf{P}(X_T \geq N^{\frac{1}{3}+\delta^{\prime}}) +  (1-N^{\frac{1}{3}+\delta^{\prime}-\delta})^{N^{\frac{1}{3}+\delta^{\prime}}} \geq 1 - \frac{1}{2}N^{-3\delta^{\prime}}
\end{align*}
for all $N$ sufficiently large. A similar statement for $Z_{i,t} <  N-N^{\frac{1}{3}+\delta^{\prime}}$ gives the desired lower bound on the probability of $B_{i,T}$. For the second statement \eqref{eq:QuickMix}, we recall \eqref{eq:BoundaryTime} in the proof of Lemma \ref{lem:MartingaleInterior} which implies that for any starting position $>N^{\frac{1}{3}+\delta^{\prime}}$, the expected time to reach some position $>k$ is of order at most $N^{-\frac{1}{3}+2\delta^{\prime}}$. Using  the same arguments as in the proof of Proposition \ref{prop:genupperbound}, together with the fact that 
\begin{equation}
\max\left( N^{-\frac{1}{3}+\delta^{\prime}} \frac{N}{k}, \frac{N^2}{k^3}\right) \leq N^{-2\delta^{\prime}} 
\end{equation}
for all $N$ large enough, we see that there exist constants $c_1,c_2>0$ such that 
\begin{equation}\label{eq:IterationCompare}
\mathbf{P}\big( \tau_{i}> T+c_1N^{-2\delta^{\prime}} \,  \big| \,  B_{i,T}\cap B_{i,T}^{\prime}\big) \leq 1- c_2 \eqpd
\end{equation} Iterating   \eqref{eq:IterationCompare} now $N^{\delta^{\prime}}$ many times gives the desired result.
\end{proof}

\begin{proof}[Proof of Proposition \ref{pro:Comparison}] Note that we obtain \eqref{eq:LargeKWith} by combining Remark \ref{rem:lowerBound} and Proposition \ref{prop:genupperbound}. The lower bound in \eqref{eq:LargeKWithout} on the mixing time of the $S_k$ shuffle follows from observing that the $\varepsilon$-mixing time is bounded from below by the time $T$ it takes such that at least one of the  cards initially at positions $1$ or $N$ has moved with probability at least $1-\varepsilon$ until time $T$. Hence, it remains to prove  \eqref{eq:UpperSlow}. Using \eqref{eq:QuickMix} in Lemma \ref{lem:CoalescenceCompare}, it suffices to show that with probability at least $1-\varepsilon/2$, there exists some $n\in \N$, some constant $C=C(\varepsilon,n)>0$, and a sequence of non-negative times $(T_{1},T_2,\dots,T_{n})$ such that 
\begin{equation}\label{eq:RequiredT}
T_{j+1}-T_{j} > N^{-\delta^{\prime}} \text{ for all } j\in [n-1] \quad \text{ and } \quad T_{n} \leq C \eqcom
\end{equation}
and with the property that for every $i\in [N]$, at least one of events $B_{i,T_j}$ occurs for some $j=j(i)\in [n]$. In order to define these times $(T_j)_{j \in [n]}$, consider for all $m\in \N$ the event $D_m$ that during the time interval $[2m-1,2m]$, both intervals $[k]$ and $\{N-k+1,\dots,N\}$ receive an update. Let $(\mathcal{F}_t)_{t \geq 0}$ denote the natural filtration under the coupling $\mathbf{P}$. Then 
\begin{equation}\label{eq:Filtration}
    \mathbf{P}(D_m \, | \, \mathcal{F}_{2m-1} )\geq \frac{1}{4} \quad \text{ and } \quad \mathbf{P}( B_{i,2m} \cap B_{i,2m}^{\prime} \, | \, D_m, \mathcal{F}_{2m-1} )\geq 1-2N^{-\delta^{\prime}}
\end{equation}
 for all $m\in \N$, uniformly in $i\in [N]$, iterating \eqref{eq:AssumptionBulk} of Lemma \ref{lem:CoalescenceCompare} for the second statement. 
Let $T_j$ be the $j^{\text{th}}$ time that the event $D_m$ occurs. Choosing $C>0$ sufficiently large, we see that for any fixed $n\in \N$, the times $(T_i)_{i\in [n]}$ satisfy \eqref{eq:RequiredT} with probability at least $1-\varepsilon/4$. To ensure that with probability at least $1-\varepsilon/4$, for every $i\in [N]$ at least one of events $B_{i,T_j}$ occurs for some $j\in [n]$, we set $n=4/\delta^{\prime}$. 
From \eqref{eq:Filtration}, we see that for every fixed $i\in [N]$, with probability at least $1-N^{-2}$, the event $B_{i,T_j} \cap B_{i,T_j}^{\prime}$ holds for some $j\in [n]$, and all $N$ large enough. Together with a union bound over all $i\in [N]$, this finishes the proof.
\end{proof}

We conclude with a conjecture on the mixing time of the $\sk$ shuffle when $k=o(N^{1/2})$. 
\begin{conjecture}\label{conj:Open}  Let $k=o(N^{1/2})$. Then for all $\varepsilon \in (0,1)$, we have that
\begin{equation}
 \lim_{N \rightarrow \infty} \frac{k(k^2-1)\cdot \tmix(\varepsilon)}{N^{2}}  = \lim_{N \rightarrow \infty} \frac{k(k^2-1)\cdot \tmix^{\prime}(\varepsilon)}{N^{2}} = \frac{6}{\pi^2}  \eqcom
\end{equation} i.e.\ the $\sk$ shuffle with and without boundary exhibits cutoff.
\end{conjecture}

\bibliographystyle{plain}
\bibliography{bib}

\begin{thebibliography}{10}

\bibitem{AlamedaBernouliLaplaceMultiSwap}
Joseph~S. Alameda, Caroline Bang, Zachary Brennan, David~P. Herzog, Jürgen
  Kritschgau, and Elizabeth Sprangel.
\newblock {Cutoff in the Bernoulli-Laplace model with $O(n)$ swaps}.
\newblock {\em preprint \url{https://arxiv.org/abs/2203.08647}}, 2022.

\bibitem{BD}
Dave Bayer and Persi Diaconis.
\newblock Trailing the dovetail shuffle to its lair.
\newblock {\em Ann. Appl. Probab.}, 2(2):294--313, 1992.

\bibitem{BBHM:MixingBias}
Itai Benjamini, Noam Berger, Christopher Hoffman, and Elchanan Mossel.
\newblock Mixing times of the biased card shuffling and the asymmetric
  exclusion process.
\newblock {\em Transactions of the American Mathematical Society},
  357(8):3013--3029, 2005.

\bibitem{BSZ}
Nathana\"{e}l Berestycki, Oded Schramm, and Ofer Zeitouni.
\newblock Mixing times for random {$k$}-cycles and coalescence-fragmentation
  chains.
\newblock {\em Ann. Probab.}, 39(5):1815--1843, 2011.

\bibitem{BN}
Megan Bernstein and Evita Nestoridi.
\newblock Cutoff for random to random card shuffle.
\newblock {\em Ann. Probab.}, 47(5):3303--3320, 2019.

\bibitem{CaputoSpectralIndependence}
Antonio Blanca, Pietro Caputo, Zongchen Chen, Daniel Parisi, Daniel
  Štefankovič, and Eric Vigoda.
\newblock {On mixing of Markov chains: coupling, spectral independence, and
  entropy factorization}.
\newblock {\em Electronic Journal of Probability}, 27(none):1–42, 2022.

\bibitem{BlancaSwendsenTrees}
Antonio Blanca, Zongchen Chen, Daniel {\v{S}}tefankovi{\v{c}}, and Eric Vigoda.
\newblock {The Swendsen--Wang dynamics on trees}.
\newblock {\em Random Structures \& Algorithms}, 2022.

\bibitem{borcea2009negative}
Julius Borcea, Petter Br{\"a}nd{\'e}n, and Thomas Liggett.
\newblock Negative dependence and the geometry of polynomials.
\newblock {\em Journal of the American Mathematical Society}, 22(2):521--567,
  2009.

\bibitem{BorceaBranden2009}
Julius Borcea and Petter Brändén.
\newblock {The Lee-Yang and Pólya-Schur programs. I. Linear operators
  preserving stability}.
\newblock {\em Inventiones mathematicae}, 177(3):541–569, Sep 2009.

\bibitem{DSHA}
Persi Diaconis and Mehrdad Shahshahani.
\newblock Generating a random permutation with random transpositions.
\newblock {\em Z. Wahrsch. Verw. Gebiete}, 57(2):159--179, 1981.

\bibitem{BL}
Persi Diaconis and Mehrdad Shahshahani.
\newblock {Time to reach stationarity in the {B}ernoulli-{L}aplace diffusion
  model}.
\newblock {\em SIAM J. Math. Anal.}, 18(1):208--218, 1987.

\bibitem{Eskenazis_Nestoridi_2020}
Alexandros Eskenazis and Evita Nestoridi.
\newblock {Cutoff for the Bernoulli–Laplace urn model with $o(n)$ swaps}.
\newblock {\em Annales de l’Institut Henri Poincaré, Probabilités et
  Statistiques}, 56(4):2621–2639, 2020.

\bibitem{ganguly2020information}
Shirshendu Ganguly and Insuk Seo.
\newblock {Information percolation and cutoff for the random-cluster model}.
\newblock {\em Random Structures \& Algorithms}, 57(3):770--822, 2020.

\bibitem{gantert2020mixing}
Nina Gantert, Evita Nestoridi, and Dominik Schmid.
\newblock Mixing times for the simple exclusion process with open boundaries.
\newblock {\em Ann. Appl. Probab.}, 33(2):972--1012, 2023.

\bibitem{GuoJerrumBlockIsing}
Heng Guo and Mark Jerrum.
\newblock {Random Cluster Dynamics for the Ising Model is Rapidly Mixing}.
\newblock {\em The Annals of Applied Probability}, 28(2):1292--1313, 2018.

\bibitem{KnopfelBlockIsing}
Holger Knöpfel, Matthias Löwe, Kristina Schubert, and Arthur Sinulis.
\newblock {Fluctuation Results for General Block Spin Ising Models}.
\newblock {\em Journal of Statistical Physics}, 178(5):1175–1200, Mar 2020.

\bibitem{Lacoin_2016}
Hubert Lacoin.
\newblock Mixing time and cutoff for the adjacent transposition shuffle and the
  simple exclusion.
\newblock {\em The Annals of Probability}, 44(2):1426–1487, Mar 2016.

\bibitem{LevinPeresWilmer}
David~Asher Levin.
\newblock {\em Markov chains and mixing times / David A. Levin, Yuval Peres,
  Elizabeth L. Wilmer.}
\newblock American Mathematical Society, Providence, R.I, 2009.

\bibitem{liggett2010continuous}
Thomas~M. Liggett.
\newblock {\em {Continuous time Markov processes: an introduction}}, volume
  113.
\newblock American Mathematical Soc., 2010.

\bibitem{longswendsen-wang}
Yun Long, Asaf Nachmias, Weiyang Ning, and Yuval Peres.
\newblock {\em {A power law of order 1/4 for critical mean field Swendsen-Wang
  dynamics}}, volume 232.
\newblock American Mathematical Society, 2014.

\bibitem{MartinelliIsingTrees}
F.~Martinelli, A.~Sinclair, and D.~Weitz.
\newblock {The Ising model on trees: boundary conditions and mixing time}.
\newblock In {\em 44th Annual IEEE Symposium on Foundations of Computer
  Science, 2003. Proceedings.}, pages 628--639, 2003.

\bibitem{MartinelliBlockIsing}
Fabio Martinelli.
\newblock {\em {Lectures on Glauber Dynamics for Discrete Spin Models}}, page
  93–191.
\newblock Springer Berlin Heidelberg, Berlin, Heidelberg, 1999.

\bibitem{nam2019cutoff}
Danny Nam and Evita Nestoridi.
\newblock Cutoff for the cyclic adjacent transposition shuffle.
\newblock {\em The Annals of Applied Probability}, 29(6):3861--3892, 2019.

\bibitem{EvitaStarTranspositions}
Evita Nestoridi.
\newblock {The Limit Profile of Star Transpositions}.
\newblock {\em preprint \url{https://arxiv.org/abs/2111.03622}}, 2021.

\bibitem{Peres_Winkler_2013}
Yuval Peres and Peter Winkler.
\newblock {Can Extra Updates Delay Mixing?}
\newblock {\em Communications in Mathematical Physics}, 323(3):1007–1016, Nov
  2013.

\bibitem{salez2022}
Justin Salez.
\newblock Universality of cutoff for exclusion with reservoirs.
\newblock {\em preprint \url{https://arxiv.org/abs/2201.03463 }}, 2022.

\bibitem{Tran2022}
Hong-Quan Tran.
\newblock {Cutoff for the non reversible SSEP with reservoirs}.
\newblock {\em preprint, \url{https://arxiv.org/abs/2211.14687}}, 2022.

\bibitem{wilson2004mixing}
David~Bruce Wilson.
\newblock {Mixing times of lozenge tiling and card shuffling Markov chains}.
\newblock {\em The Annals of Applied Probability}, 14(1):274--325, 2004.

\bibitem{seoyeonblockising}
Seoyeon Yang.
\newblock {Cutoff and Dynamical Phase Transition for the General
  Multi-component Ising Model}.
\newblock {\em preprint \url{https://arxiv.org/abs/2112.04976}}, 2021.

\bibitem{LingfuBiasedCardShuffling}
Lingfu Zhang.
\newblock {Cutoff profile of the Metropolis biased card shuffling}.
\newblock {\em preprint \url{https://arxiv.org/abs/2208.13383}}.

\end{thebibliography}

\subsection*{Acknowledgment}
This project was initiated at the SAMSI Virtual Workshop in Combinatorial probability, which was supported by NSF Grant DMS-1929298.
EN was supported by NSF Grant DMS-2052659. AP was supported by the University of Texas at Austin Dean's Strategic Fellowship. DS acknowledges the DAAD PRIME program for financial support. We thank Jonathan Hermon for helpful comments.
\end{document}